\documentclass[11pt, letterpaper]{amsart}
\usepackage{amsfonts,amsthm,mathtools,mathrsfs,caption}
\usepackage[all,cmtip]{xy}
\usepackage{fullpage}
\usepackage{amsmath}
\usepackage{amssymb}
\usepackage{enumerate}
\usepackage{tikz}
\usepackage[backref]{hyperref}
\usepackage{cleveref}
\usepackage{verbatim}
\usepackage{cancel}
\usepackage{float}

\newtheorem{lem}{Lemma}[section]
\newtheorem{thm}[lem]{Theorem}
\newtheorem{proposition}[lem]{Proposition}
\newtheorem{prop}[lem]{Proposition}
\newtheorem{cor}[lem]{Corollary}

\theoremstyle{definition}
\newtheorem{remark}[lem]{Remark}
\newtheorem{ex}{Example}

\DeclareMathAlphabet{\curly}{U}{rsfs}{m}{n}

\newcommand{\End}{\operatorname{End}}
\newcommand{\Aut}{\operatorname{Aut}}
\newcommand{\Gal}{\operatorname{Gal}}

\newcommand{\ord}{\operatorname{ord}}

\newcommand{\Q}{\mathbb{Q}}

\newcommand{\Z}{\mathbb{Z}}

\newcommand{\GL}{\operatorname{GL}}

\newcommand{\OO}{\mathcal{O}}
\newcommand{\0}{\mathcal{O}}

  \newcommand{\im}{\operatorname{im}}

\mathchardef\mhyphen="2D

\title{Minimal Torsion Curves in Geometric Isogeny Classes}

\author{Abbey Bourdon}
\address{Wake Forest University, Winston-Salem, NC 27109, USA}
\email{bourdoam@wfu.edu}
\urladdr{http://users.wfu.edu/bourdoam/}

\author{Nina Ryalls}
\address{University of Georgia, Athens, GA 30602 USA}
\email{Nina.Ryalls@uga.edu}

\author{Lori D. Watson}
\address{Trinity College, Hartford, CT 06106 USA}
\email{lori.watson@trincoll.edu}
\urladdr{https://loridwatson.com}

\begin{document}

\begin{abstract} 
In this paper, we introduce the study of minimal torsion curves within a fixed geometric isogeny class. For a $\overline{\Q}$-isogeny class $\mathcal{E}$ of elliptic curves and $N \in \Z^+$, we wish to determine the least degree of a point on the modular curve $X_1(N)$ associated to any $E \in \mathcal{E}$. In the present work, we consider the cases where $\mathcal{E}$ is \emph{rational}, i.e., contains an elliptic curve with rational $j$-invariant, or where $\mathcal{E}$ consists of elliptic curves with complex multiplication (CM). If $N=\ell^k$ is a power of a single prime, we give a complete characterization upon restricting to points of odd degree, and also in the case where $\mathcal{E}$ is CM. We include various partial results in the more general setting. 
\end{abstract}

\maketitle

\section{Introduction}

%It is a consequence of a theorem of Shafarevich that there are only finitely many elliptic curves over a number field $F$ contained in any $F$-rational isogeny class, and each elliptic curve $E/F$ has a finite torsion subgroup by the Mordell-Weil Theorem. Prior work has investigated how the torsion subgroup varies within an $F$-rational isogeny class; see, for example, \cite{katz81,FN07,ross94,ChiloyanLR}. In this paper, we instead study torsion points of elliptic curves within a fixed \emph{geometric} isogeny class. Here, two elliptic curves lie in the same geometric isogeny class if they are connected by an isogeny defined over $\overline{\Q}$. Any geometric isogeny class contains infinitely many elliptic curves, up to $\overline{\Q}$-isomorphism, and there is no upper bound on the size of a torsion subgroup. Instead, we introduce the problem of studying low degree points associated to $\mathcal{E}$ on modular curves. Motivated by ties to Serre's Uniformity Problem \cite[Theorem 1.2]{OddDegQCurve}, we focus on the modular curve $X_1(N)$, whose non-cuspidal points parametrize elliptic curves with a distinguished point of order $N$. We recall a more precise definition of this curve and the degree of its points in  $\S2.4$.

%For a given geometric isogeny class $\mathcal{E}$ of elliptic curves, our central questions are as follows:

The modular curve $X_1(N)$ is an algebraic curve over $\Q$ whose non-cuspidal points parameterize elliptic curves with a distinguished point of order $N$. Understanding degree $d$ points\footnote{Recall that the degree of a closed point $x$ on a curve $C$ over a number field $k$ is the degree of its residue field: $[k(x):k]$. This is equivalent to the size of the Galois orbit of points in $C(\overline{k})$ associated to $x$.} on $X_1(N)$ is an essential step in classifying torsion subgroups of elliptic curves defined over number fields of degree $d$ --- a problem which is solved only for $d \leq 4$  \cite{mazur77,kamienny86,KM88,kamienny92,Deg3Class,classificationdeg4points}. By the Riemann-Roch Theorem, the curve $X_1(N)$ has infinitely many points in any degree greater than its genus, so it suffices to characterize points of low degree. One important class of low degree points are \textbf{sporadic points}, which are $x\in X_1(N)$ for which there are only finitely many points of degree at most $\deg(x)$. Ruling out unexpected sporadic points on $X_1(N)$ would have implications for major open questions in the field, including Serre's Uniformity problem (as in \cite[Theorem 1.3]{OddDegQCurve}) and uniformity conjectures for the modular curve $X_0(N)$; see \cite[Conjecture 18]{ogg} and \cite[Conjecture 1.1]{QuadPointsConjecture2024}. Such applications further motivate the need for an improved understanding of low degree points

%For example, a negative answer to Serre's Uniformity problem would imply the existence of unexpected sporadic points on $X_1(\ell^2)$ for arbitrarily large primes $\ell$, as shown in \cite[Theorem 1.3]{OddDegQCurve}. Recent conjectures concerning quadratic points on the modular curve $X_0(N)$, as in \cite[Conjecture 18]{ogg} and \cite[Conjecture 1.1]{QuadPointsConjecture2024}, are also related as a quadratic point on $X_0(N)$ will lift to a sporadic point on $X_1(N)$ for $N$ sufficiently large.

%or a recent conjecture of Balakrishnan and Mazur \cite[Conjecture 18]{ogg} 
%This further motivates the need for an improved understanding of low degree points.

While studying low degree points on $X_1(N)$ is quite difficult in general, the problem can be made more tractable by restricting the class of points under consideration. One line of investigation considers only points associated to elliptic curves with complex multiplication (CM); see, for example, \cite{BC2,LeastCMdegree,ClarkVolcanoes}. The least degree of a point on $X_1(N)$ associated to an elliptic curve with CM by an order $\OO$ in an imaginary quadratic field $K$ is given in \cite[Theorem 1.2]{BC2}. One could more generally hope to characterize the least degree of a point on $X_1(N)$ as we range over elliptic curves with CM by \emph{any} order in $K$. Since an elliptic curve with CM by $\mathcal{O} \subseteq K$ is isogenous over $\overline{\Q}$ to one with CM by the full ring of integers, this is in essence a special case of the following question: For a given geometric isogeny class $\mathcal{E}$ of elliptic curves and fixed $N\in \Z^+$, what is the least degree of a point on $X_1(N)$ associated to an elliptic curve in $\mathcal{E}$? And what elliptic curve(s) in $\mathcal{E}$ attain a point on $X_1(N)$ of least possible degree? We investigate these questions in the present work.

%Prior work of the first author and Clark \cite{BC2} gives the least degree of a point on $X_1(N)$ associated to an elliptic curve with CM by a fixed order in an imaginary quadratic field, while work of Clark, Genao, Pollack, and Saia \cite{LeastCMdegree} investigates the least degree of any CM point on $X_1(N)$. Our results fall somewhat in the middle of these two directions of study -- we investigate the least degree of a CM point across all orders within a fixed imaginary quadratic field. 

Let $\mathcal{E}$ be a geometric isogeny class of elliptic curves. That is, for any $E_1,E_2 \in \mathcal{E}$, there exists an isogeny $\varphi: E_1 \rightarrow E_2$ defined over $\overline{\Q}$. We want to find the minimum $d$ such that there exists $F$ of degree $d$ and $E/F \in \mathcal{E}$ such that $E(F)$ contains a point $P$ of order $N$. We call $E$ a \textbf{minimal torsion curve for $\mathcal{E}$ of level $N$}. Because $F$ is of minimal degree, the residue field of the closed point associated to $(E,P)\in X_1(N)(\overline{\Q})$ is also of degree $d$, and so we may alternatively view $E$ as a curve over $\overline{\Q}$ identified by its $j$-invariant, $j_{min}$. In this paper, we study minimal torsion curves of prime-power level. For CM isogeny classes, we obtain a near complete characterization.

\begin{thm}\label{CMthm}
Let $\mathcal{E}$ be a $\overline{\Q}$-isogeny class of elliptic curves with CM by an order in the imaginary quadratic field $K$. For a prime number $\ell$, the the least degree of a point on $X_1(\ell^k)$ associated to any $E \in \mathcal{E}$ is given in Propositions \ref{CaseSplit}, \ref{CaseInert}, and \ref{CaseRamified}. If $\ell$ is split in $K$, then an elliptic curve with CM by the full ring of integers in $K$ is a minimal torsion curve for $\mathcal{E}$ of level $\ell^k$ and $[\Q(j_{min}):\Q]=h_K$, the class number of $K$. Otherwise $[\Q(j_{min}):\Q] \rightarrow \infty$ as $k \rightarrow \infty$.
\end{thm}

If $\mathcal{E}$ is a $\overline{\Q}$-isogeny class of elliptic curves with CM by an order in $K$, then $[\Q(j(E)):\Q] \geq h_K$ for any $E \in \mathcal{E}$. Thus when $\ell$ is split in $K$, a minimal torsion curve has $[\Q(j_{min}):\Q]$ minimal for $\mathcal{E}$. However, if $\ell$ is inert or ramified in $K$, the extension $[\Q(j_{min}):\Q]$ will generally not be minimal for $\mathcal{E}$, and indeed must grow with $k$. These minimal torsion curves have exceptional arithmetic, such as the existence of a high-power $\ell^k$-isogeny defined over $\Q(j_{min})$, which more than compensates for the degree of $\Q(j_{min})$. This phenomenon was previously observed for CM points of odd degree in \cite[Remark 2.7]{BP}. We illustrate with the following example.
\begin{ex}
Let $\mathcal{E}$ be the $\overline{\Q}$-isogeny class of $E/\Q$ with LMFDB label \href{https://www.lmfdb.org/EllipticCurve/Q/27/a/1}{27.a1}. Here $K=\Q(\sqrt{-3})$, and we take $\ell=3$ which is ramified in $K$. There are two rational $j$-invariants associated to elliptic curves in $\mathcal{E}$, namely $-12288000$ and $0$. Thus $1$ is the minimal degree of an extension generated by the $j$-invariant of $E \in \mathcal{E}$. However, by Proposition \ref{CaseRamified}, the $j$-invariant of a minimal torsion curve of level $3^k$ generates an extension of degree at least $3^{(k-5)/2}$. Thus there is no minimal torsion curve with rational $j$-invariant for $k \geq 6$, and the degree of $\Q(j_{min})$ must grow with $k$.
\end{ex}

Even in non-CM isogeny classes, minimal torsion curves may have $j$-invariant generating a field of unexpectedly large degree. In some cases, these minimal torsion curves improve degree bounds arising from elliptic curves having rational $j$-invariant.

\begin{ex}
Let $\mathcal{E}$ be the $\overline{\Q}$-isogeny class containing $E/\Q$ with LMFDB label \href{https://www.lmfdb.org/EllipticCurve/Q/9225/l/1}{9225.l1}. Here $j(E)\in \Q$, but in Section $8$ we see any minimal torsion curve for $\mathcal{E}$ of level 49 has $[\Q(j_{min}):\Q] \geq 2$. This class is notable since $\mathcal{E}$ contains an elliptic curve giving a point on $X_1(49)$ of degree $42$, which is lower than the degree of any $x \in X_1(49)$ with $j(x)\in \Q$. See Corollary \ref{LowerDegCor}.
\end{ex}

Our next results concern the case where $\mathcal{E}$ is \textbf{rational}, i.e., contains an elliptic curve with $j$-invariant in $\Q$. For non-CM classes, we obtain a characterization of minimal torsion curves upon restriction to points of odd degree, strengthening \cite[Proposition 4.1]{OddDegQCurve}. 

\begin{thm} \label{Thm1.3}
Let $\mathcal{E}$ be a rational $\overline{\Q}$-isogeny class of elliptic curves which is non-CM. If $E \in \mathcal{E}$ and $x=[E,P]\in X_1(\ell^k)$ is a point of odd degree, then $\ell \in \{2,3,5,7,11,13\}$. The least odd degree point on $X_1(\ell^k)$ associated to $E' \in \mathcal{E}$ is given in Propositions \ref{BestDivLargePrimes}, \ref{DivConditionsPrime3} and \ref{2case}, while the following divisibility conditions hold and are best-possible across all such $\mathcal{E}$:

\begin{enumerate}
\item If $\ell =13$, then $3 \cdot 13^{2k-2} \mid \deg(x)$.
\item If $\ell =11$, then $5 \cdot 11^{2k-2} \mid  \deg(x)$.
\item If $\ell=7$ and $\mathcal{E}$ does not contain $E'$ with $j(E')= 3^3\cdot5\cdot7^5/2^7$, then $7^{2k-2} \mid  \deg(x)$.
\item If $\ell=7$ and $\mathcal{E}$ contains $E'$ with $j(E')= 3^3\cdot5\cdot7^5/2^7$, then $9 \cdot 7^{\max(0,2k-3)} \mid  \deg(x)$.
\item If $\ell=5$, then $5^{\max(0,2k-3)} \mid  \deg(x)$.
\item If $\ell=3$, then $3^{\max(0,2k-4)} \mid  \deg(x)$.
\item If $\ell=2$, then $k \leq 3$ and $1 \mid \deg(x)$.
\end{enumerate}
Moreover, among odd degree points on $X_1(\ell^k)$ coming from $E' \in \mathcal{E}$, a point of least odd degree can always be associated to $E_{min} \in \mathcal{E}$ with $j(E_{min}) \in \Q$ or which is $\ell$-isogenous to an elliptic curve having rational $j$-invariant. 
\end{thm}
\begin{remark}
The exceptional class for $\ell=7$ has also been identified in another context. By Sutherland \cite{Sutherland12}, elliptic curves $E/\Q$ with $j(E)=3^3\cdot5\cdot7^5/2^7$ provide the only counterexamples over $\Q$ to a local-global principle for rational isogenies of prime degree.
\end{remark}

In particular, Theorem \ref{Thm1.3} states that under the given assumptions, there exists a minimal torsion curve for $\mathcal{E}$ of level $\ell^k$ which is at most $\ell$-isogenous to an elliptic curve having rational $j$-invariant. Away from points of odd degree, Proposition \ref{LevelEllImages} shows the same condition holds for $\mathcal{E}$ containing a rational elliptic curve with $\ell$-adic Galois representation of level $\ell$. These are special cases of the following result, which is a consequence of Serre's Open Image Theorem \cite{serre72}.

\begin{thm}\label{IntroPropSerre}
Let $\mathcal{E}$ be a rational $\overline{\Q}$-isogeny class of non-CM elliptic curves, and let $\ell$ be prime. There exists a constant $C=C(\mathcal{E},\ell)$ such that for any $k \in \Z^+$ a point of least degree on $X_1(\ell^k)$ coming from $E' \in \mathcal{E}$ can be associated to $j_{min}\in \mathcal{E}$ which is $d$-isogenous to a rational $j$-invariant for some $d\leq C$.
\end{thm}

Thus unlike the CM case addressed in Theorem \ref{CMthm}, there always exists a minimal torsion curve of level $\ell^k$ with $j$-invariant in an extension of bounded degree. Unfortunately, our proof does not make the constant $C$ explicit; see Theorem \ref{PropMinTorsionCurveBound}. If $\mathcal{E}$ contains a non-CM elliptic curve $E/\Q$ with $\ell$-adic Galois representation of level $\ell$, then $C=\ell$ by Proposition \ref{LevelEllImages}. However, we fail to identify a more general connection between $C$ and the level of $E/\Q \in \mathcal{E}$, though it is natural to ask whether one exists (Remark \ref{RemarkExplicitC}).

%If $\mathcal{E}$ is a class of elliptic curves with complex multiplication, then Serre's Open Image Theorem does not apply. In this case, $\mathcal{E}$ consists of elliptic curves with CM by various orders within a fixed imaginary quadratic field $K$. Any elliptic curve $E$ with CM by the full ring of integers in $K$ will have $[\Q(j(E)):\Q]$ minimal for $\mathcal{E}$, and we investigate the isogeny distance from $E$ to a minimal torsion curve with $j$-invariant $j_{min}$. 

\subsection{General Approach.} If $\mathcal{E}$ is a rational geometric isogeny class of non-CM elliptic curves, then for any $E \in \mathcal{E}$ there exists an isogeny $\varphi: E \rightarrow E_0$ defined over $\overline{\Q}$, where $E_0$ is the base extension of an elliptic curve defined over $\Q$. Up to replacing $E$ by a quadratic twist if necessary, we may assume $E$, $E_0$, and the isogeny $\varphi$ are defined over a number field $F$. A key observation is that over $F$, the image of both $\ell$-adic Galois representations have the same index in $\GL_2(\Z_{\ell})$:
\[
[\GL_2(\Z_{\ell}):\im \rho_{E/F, \ell^{\infty}}]=[\GL_2(\Z_{\ell}):\im \rho_{E_0/F, \ell^{\infty}}].
\]
See, for example, \cite[Proposition 2.1.1]{greenberg2012}. This can be leveraged to give lower bounds on the degree of a point on $X_1(\ell^k)$ associated to $\mathcal{E}$ in terms of $[\GL_2(\Z_{\ell}):\im \rho_{E_0/\Q, \ell^{\infty}}]$. We obtain the following proposition, which strengthens \cite[Lemma 4.6]{OddDegQCurve}.

\begin{prop}\label{Prop1.5intro}
Let $\mathcal{E}$ be a rational $\overline{\Q}$-isogeny class of non-CM elliptic curves. Suppose $\ell$ is a prime number and $k \in \Z^+$. There exists $E_0/\Q \in\mathcal{E}$ and $x \in X_1(\ell)$ with $j(x)=j(E_0)$ such that the degree of any point on $X_1(\ell^k)$ associated to $\mathcal{E}$ is divisible by 
\[\delta \coloneqq \begin{cases}
\deg(x)\cdot \ell^{\max(0,2k-2-d)} \text{ if $\ell$ is odd},\\
\deg(x) \cdot \ell^{\max(0,2k-3-d)} \text{ if $\ell=2$},
\end{cases}\]
where $d \coloneqq \ord_{\ell}([\GL_2(\Z_{\ell}): \im \rho_{E_0/\Q, \ell^{\infty}}])$.
\end{prop}

In some cases, we show these lower bounds are best-possible by explicitly constructing $E' \in \mathcal{E}$ giving a degree $\delta$ point on $X_1(\ell^k)$. A natural example is when $d=0$ and $\ell$ is odd, in which case a twist of $E_0$ is a minimal torsion curve for $X_1(\ell^k)$; if $d>0$, it may be that $j(E') \notin \Q$. In other instances, the lower bounds of Proposition \ref{Prop1.5intro} can be strengthened by showing that a point on $X_1(\ell^k)$ in degree $\delta$ would produce a subgroup of $\im \rho_{E_0/\Q,\ell^{\infty}}$ which does not occur; see Proposition \ref{Prop5.7}. Throughout we make use of the partial classification of images of $\ell$-adic Galois representations of elliptic curves over $\Q$ due to Rouse and Zureick-Brown \cite{RouseDZB} and Rouse, Sutherland, and Zureick-Brown \cite{RSZ21}. 

\begin{remark} Let $\mathcal{E}$ be a rational $\overline{\Q}$-isogeny class of non-CM elliptic curves. By Theorem \ref{IntroPropSerre}, any minimal torsion curve for $X_1(\ell^k)$ is at most $C$-isogenous to an elliptic curve with rational $j$-invariant for some $C$ which does not depend on $k$. Thus there exists a finite set of $j$-invariants $\mathcal{J}=\{j_1, j_2, \dots, j_r\}$ such that for any $k \in \Z^+$, there exists a minimal torsion curve for $\mathcal{E}$ of level $\ell^k$ with $j$-invariant in $\mathcal{J}$. However, our proof does not make $\mathcal{J}$ explicit, so our results may not be obtained by checking a finite number of $j$-invariants. On the other hand, for a fixed $k \in \Z^+$, a finite check is possible since the $j$-invariant of any minimal torsion curve must generate an extension degree at most $\deg(X_1(\ell^k) \rightarrow X_1(1))$.
\end{remark}

Results concerning CM isogeny classes rely heavily on prior work of the first author in collaboration with Clark \cite{BC1,BC2}.

\subsection{Other Related Work.} Cremona and Najman \cite{CremonaNajmanQCurve} prove numerous results concerning torsion points of $\Q$-curves defined over number fields of odd degree. Any such elliptic curve is necessarily isogenous to one having rational $j$-invariant, providing immediate ties to our study of minimal torsion curves in rational geometric isogeny classes. This class of $\Q$-curves is again studied in \cite{OddDegQCurve}, where the first author and Najman show the $\overline{\Q}$-isogeny class containing the elliptic curve with $j$-invariant $-140625/8$ is the unique rational non-CM class giving rise to a sporadic point of odd degree on any modular curve $X_1(N)$. More recent work of Genao \cite{Genao21} provides ``typical" bounds on the size of the torsion subgroup of an elliptic curve over a number field which belongs to a rational $\overline{\Q}$-isogeny class, and his subsequent work \cite{Genao22} pursues polynomial bounds on such torsion subgroups.

Prior work of the first author and Clark \cite{BC2} gives the least degree of a point on $X_1(N)$ associated to an elliptic curve with CM by a fixed order in an imaginary quadratic field, while work of Clark, Genao, Pollack, and Saia \cite{LeastCMdegree} investigates the least degree of any CM point on $X_1(N)$. Our results fall somewhat in the middle of these two directions of study -- we investigate the least degree of a CM point across all orders within a fixed imaginary quadratic field. 

\subsection{Code.} We make frequent use of the computer algebra system Magma \cite{Magma}. All code is available at \url{https://github.com/abbey-bourdon/minimal_torsion_curves}.

\section*{Acknowledgements} We thank John Cremona, \'{A}lvaro Lozano-Robledo, Jeremy Rouse, Parker Schwartz, and Andrew Sutherland for helpful conversations. We thank John Cremona, Michael Eddy, Tyler Genao, and Filip Najman for helpful comments on an earlier version of this paper. We are especially grateful to the anonymous referee, who made many comments and suggestions with greatly improved the exposition. All authors were supported by NSF grant DMS-2137659. The first author was partially supported by an A. J. Sterge Faculty Fellowship and NSF grant DMS-2145270.

\section{Background and Notation}

\subsection{Conventions.} Throughout, $F$ denotes a number field and $\overline{F}$ denotes a fixed algebraic closure of $F$. We write $\Gal_F$ for the absolute Galois group $\Gal(\overline{F}/F)$.

For an elliptic curve $E$ defined over $F$ and $N\in\Z^+$, the collection of all points in $E(\overline{F})$ of order dividing $N$ is denoted $E[N]$. This is a free $\Z/N\Z$-module of rank 2. Any elliptic curve $E/F$ corresponds to an equation of the form $y^2=x^3+Ax+B$, and we can define its $j$-invariant to be $j(E) \coloneqq 1728 \frac{4A^3}{4A^3+27B^2}$. This element of $F$ characterizes $E$ up to $\overline{F}$-isomorphism, and we call any $E'$ with $j(E')=j(E)$ a twist of $E$.

For a prime number $\ell$, our notation for subgroups of $\GL_2(\Z/\ell\Z)$ and $\GL_2(\Z_{\ell})$ follows \cite{sutherland} and \cite{RSZ21}, respectively. For $\ell$-adic images, this is known as the ``RSZB label" and has the form \texttt{N.i.g.n}, where \texttt{N} is the level, \texttt{i} is the index, \texttt{g} is the genus, and \texttt{n} is a positive integer used to distinguish nonconjugate subgroups. We also refer to specific elliptic curves over $\Q$ by their $L$-functions and Modular Forms Database \cite{LMFDB} (LMFDB) label.

We always view the modular curve $X_1(N)$ as an algebraic curve over $\Q$; see $\S\ref{Section2.4}$ for details. By closed point, we mean a $\Gal_{\Q}$-orbit of points in $X_1(N)(\overline{\Q})$. If $x \in X_1(N)$ is closed, we define the degree of $x$ to be the degree of its residue field $\Q(x)$. By taking the sum of Galois conjugates, such a closed point of degree $d$ can be viewed as an irreducible $\Q$-rational effective divisor of degree $d$.

\subsection{Galois Representations}\label{Section2.2} Let $E$ be an elliptic curve defined over a number field $F$, and let $\ell$ be a prime number. For any $k \in \Z^+$, the elements of $\Gal_F$ induce natural automorphisms of the points in $E(\overline{F})$ of order dividing $\ell^k$, denoted $E[\ell^k]$. This action is recorded in the \textbf{mod $\ell^k$ Galois representation} associated to $E$, 
\[
\rho_{E,\ell^k}: \Gal_F \rightarrow \Aut(E[\ell^k]) \cong \GL_2(\Z/\ell^k\Z).
\]
By choosing compatible bases as $k$ ranges over all positive integers, the mod $\ell^k$ Galois representations fit together to give the \textbf{$\ell$-adic Galois representation} associated to $E$, 
\[
\rho_{E,\ell^{\infty}}: \Gal_F \rightarrow \GL_2(\Z_{\ell}),
\]
which encodes the Galois action on all points in $E(\overline{F})$ of order a power of $\ell$. If $E$ is non-CM, then $\im \rho_{E,\ell^{\infty}}$ is an open subgroup of $\GL_2(\Z_{\ell})$ by Serre's Open Image Theorem \cite{serre72}. Thus there exists a nonnegative integer $d$ such that $\im \rho_{E,\ell^{\infty}}=\pi^{-1}(\im \rho_{E,\ell^d})$, where $\pi:\GL_2(\Z_{\ell}) \rightarrow \GL_2(\Z/\ell^d\Z)$ is the natural reduction map. The smallest such $\ell^d$ for which this holds is called the \textbf{level} of the $\ell$-adic Galois representation.

Aside from $\GL_2(\Z/\ell\Z)$, all groups which are known to occur as the image of the mod $\ell$ Galois representation associated to a non-CM elliptic curve over $\Q$ appear (up to conjugacy) in Tables 3 and 4 of \cite{sutherland}. This list is complete for $\ell \leq 13$ by work of Zywina \cite{ZywinaImages} and Balakrishnan, Dogra, M\"{u}ller, Tuitman, and Vonk \cite{Balakrishnan}, and it has been conjectured to be complete for all $\ell$ by both Sutherland \cite[Conjecture 1.1]{sutherland} and Zywina \cite[Conjecture 1.12]{ZywinaImages}. Unconditionally, we have the following result for $\ell \geq 17$.

\begin{thm}[Mazur \cite{mazur78}, Serre \cite{Serre81}, Bilu, Parent, and Rebolledo \cite{BPR13}]\label{ThmRemainingCases} 
Suppose $E/\Q$ is a non-CM elliptic curve and $\ell \geq 17$ is prime. If $\im \rho_{E,\ell}$ is not equal to $\GL_2(\Z/\ell\Z)$ and not conjugate to a group in Table 3 or 4 of \cite{sutherland}, then $\im \rho_{E,\ell}$ is contained in $C_{ns}^+(\ell)$, the normalizer of a non-split Cartan subgroup of $\GL_2(\Z/\ell\Z)$.
\end{thm}

Refinements of Theorem \ref{ThmRemainingCases} appear in \cite[Proposition 1.13]{ZywinaImages}, which is also proven in \cite[Appendix B]{LeFournLemos}. Recently, Furio and Lombardo \cite{FurioLombardo23} have shown that proper subgroups of the normalizer of a non-split Cartan subgroup do not occur for primes larger than 37. Taken together with prior work, \cite[Theorem 1.6]{FurioLombardo23} implies the following result.

\begin{thm}[Furio, Lombardo \cite{FurioLombardo23}]
Suppose $E/\Q$ is a non-CM elliptic curve and $\ell \geq 17$ is prime. If $\im \rho_{E,\ell}$ is not equal to $\GL_2(\Z/\ell\Z)$ and not conjugate to a group in Table 4 of \cite{sutherland}, then $\im \rho_{E,\ell}$ is conjugate to $C_{ns}^+(\ell)$.
\end{thm}

Many partial classification results exist for the image of the $\ell$-adic Galois representation of an elliptic curve $E/\Q$. First suppose $E$ is non-CM. The groups which occur as $\im \rho_{E,2^{\infty}}$ are known due to work of Rouse and Zureick-Brown \cite{RouseDZB}. For odd primes $\ell$, Sutherland and Zywina \cite{SutherlandZywina} have identified the images that occur infinitely often, and work of Rouse, Sutherland, and Zureick-Brown \cite{RSZ21} provides additional classification results for $3 \leq \ell \leq 11$ which are complete up to computing rational points on 6 remaining modular curves. If $E$ has complex multiplication, see work of Lozano-Robledo \cite{LRcm}. 

\subsection{Elliptic Curves with an Isogeny}

Suppose $E/F$ is an elliptic curve with an $F$-rational cyclic $N$-isogeny for $N\in \Z^+$, which means there is a cyclic subgroup of order $N$ fixed (as a group) by $\Gal_F$. {Thus there exists $P \in E(\overline{F})$ of order $N$ such that for any $\sigma \in \text{Gal}_F$, there is some $\alpha \in (\Z/N\Z)^\times$ for which  $\sigma(P) = \alpha P$.} This defines a homomorphism called the \textbf{isogeny character}: 
\begin{align*}
\chi : \Gal_F &\rightarrow (\mathbb{Z}/N\mathbb{Z})^\times \\
\sigma &\mapsto \alpha.
\end{align*} 
If the image of $\chi$ lands in $\{\pm 1\}$, then there is a twist of $E$ for which the point corresponding to $P$ becomes $F$-rational. In general, we have the following proposition.

\begin{prop} \label{IsogenyToRationalPoints}
Let $N\geq 3$ be an integer, and let $E/F$ be an elliptic curve with an $F$-rational cyclic isogeny of degree $N$. There is an abelian extension $L/F$ with $[L:F] \mid \frac{\varphi(N)}{2}$ and a quadratic twist $E'$ of $E/L$ such that $E'(L)$ has a point of order $N$.
\end{prop}

\begin{proof}
This is a consequence of \cite[Theorem 5.5]{BCS}.
\end{proof}

\subsection{Modular Curves}\label{Section2.4} In this paper, we are interested in characterizing degrees of points on the modular curve $X_1(N)$, where $N$ is a positive integer. Recall $X_1(N)$ is an algebraic curve over $\Q$ whose non-cuspidal points correspond to isomorphism classes of elliptic curves with a distinguished point of order $N$. If $E$ is an elliptic curve defined over a number field $F$ with $P\in E(F)$ of order $N$, then $(E,P)$ gives an $F$-valued point on $X_1(N)$ via this moduli interpretation. By definition, this is a morphism of $\Q$-schemes $f:\text{Spec}\, F \rightarrow X_1(N)$, and the image of $f$ is the associated closed point, denoted $[E,P]$. See \cite[Section 7.7]{modular}, \cite{DiamondIm}, \cite[$\S6.7$]{shimura}, \cite[Appendix C, $\S13$]{silverman}, or \cite{DR73} for more details. If $x\in X_1(N)$ is closed point, we define the \textbf{degree} of $x$ to be the degree of the residue field $\Q(x)$. For a non-cuspidal point $x$, we can construct $\Q(x)$ explicitly via the following result.

\begin{lem}\label{lem:ResidueField}
Let $E/\overline{\Q}$ be an elliptic curve and let $P \in E(\overline{\Q})$ be a point of order $N$. Then the residue field of the closed point $x=[E,P] \in X_1(N)$ is given by
\[
\Q(x)=\Q(j(E),\mathfrak{h}(P)),
\]
where $\mathfrak{h}: E \rightarrow E/\Aut(E) \cong \mathbb{P}^1$ is a Weber function for $E$. There is Weierstrass equation for $E$ defined over $\Q(x)$ for which $P \in E(\Q(x))$, and $\Q(x)$ is contained in any number field over which both $E$ and $P$ are defined.
\end{lem}

\begin{proof}
See, for example, \cite[Lemma 2.5]{OddDegQCurve}, and \cite[p. 274, Proposition VI.3.2]{DR73}.
\end{proof}

\begin{remark} If $E/\Q(j(E))$ corresponds to an equation of the form $y^2=x^3+Ax+B$ and $P=(x,y) \in E$, then we may take
\[ \mathfrak{h}(P) = \begin{cases} x & AB \neq 0 \\ x^2 & B = 0 \\ x^3 & A = 0 \end{cases}. \] Note $AB \neq 0
$ if $E$ is non-CM. Thus by Lemma \ref{lem:ResidueField} we can compute the degree of a closed point on $X_1(N)$ associated to a non-CM elliptic curve by factoring division polynomials. See \cite[p. 107]{shimura} for details, including a formulation of the Weber function which is more clearly model-independent.
\end{remark}

Many of our results rely on first constructing an explicit point $x \in X_1(\ell^k)$ for some small integer $k$, and then obtaining information on the degree of lifts of $x$ using formulas for the degree of maps between modular curves.
            \begin{prop}\label{prop:Degree}
                For positive integers $a$ and $b$, there is a $\Q$-rational map $f:X_1(ab) \rightarrow X_1(a)$ which sends $[E,P]$ to $[E,bP]$. Moreover
                \[
                    \deg(f)=
                    c_{f}\cdot b^2 \prod_{p \mid b,\, p \nmid a}
                    \left(1-\frac{1}{p^2}\right),
                \]
                where $c_{f}=1/2$ if $a \leq 2$ and $ab>2$, and $c_{f}=1$ otherwise. 
            \end{prop}
\begin{proof}
The moduli interpretation ensures the map is defined over $\Q$, and the degree calculation follows from \cite[p.66]{modular}.
\end{proof}

\subsection{Complex Multiplication}\label{Section2.5} An elliptic curve $E$ defined over a number field $F$ has \textbf{complex multiplication (CM)} if the ring of endomorphisms of $E$ defined over $\overline{F}$ is strictly larger than $\Z$. In this case, $\End_{\overline{F}}(E) \cong \mathcal{O}$, an order in an imaginary quadratic field $K$. We have $\mathcal{O}=\Z+\mathfrak{f}\mathcal{O}_K$ where $\mathcal{O}_K$ is the ring of integers in $K$ and $\mathfrak{f}$ is a positive integer called the \textbf{conductor} of $\mathcal{O}$. If $\mathfrak{f}=1$, then $\mathcal{O}=\mathcal{O}_K$, the \textbf{maximal order} in $K$. Each imaginary quadratic order is uniquely identified by its \textbf{discriminant},
\[
\Delta=\Delta(\mathcal{O})=\mathfrak{f}^2\Delta_K,
\]
where $\Delta_K$ is the discriminant of $K$. We have $\#\mathcal{O}^{\times}=2$ unless $\Delta=-3$ or $-4$ in which case $\#\mathcal{O}=6$ or $4$, respectively.  If $E$ has CM by the order $\mathcal{O}$ in $K$, then $[\Q(j(E)):\Q]=h(\mathcal{O})$ by \cite[Theorem 11.1]{cox}, the class number of $\mathcal{O}$. If $\mathfrak{f}=1$ then $h(\mathcal{O})=h_K$, the class number of $K$. If $\mathfrak{f}>1$, then by \cite[Corollary 7.24]{cox}
\begin{equation} \label{ClassNo}
h(\OO) = [\Q(j(E)):\Q]=h_K \frac{\mathfrak{f} }{[\mathcal{O}_K^{\times}:\mathcal{O}^{\times}]} \prod_{p \mid \mathfrak{f}} \left( 1- \left(\frac{\Delta_K}{p} \right) \frac{1}{p} \right).
\end{equation} 
Any two $j$-invariants of $\mathcal{O}$-CM elliptic curves are Galois conjugate algebraic integers.

\section{Preliminary Results}

In this section, we begin by establishing a brief technical result concerning the field of definition of an isogeny ($\S\ref{Section3.1}$), which essential follows from prior work of Cremona and Najman \cite[Corollary A.5]{CremonaNajmanQCurve} or Clark \cite[Proposition 3.2]{ClarkVolcanoes}. This is used in $\S\ref{Section3.2}$ to prove a general divisibility condition for points on modular curves corresponding to a fixed rational $\overline{\Q}$-isogeny class of non-CM elliptic curves, strengthening \cite[Lemma 4.6]{OddDegQCurve}. In $\S\ref{Section3.3}$, we conclude with a lemma concerning the image of Galois representations attached to elliptic curves connected by a rational cyclic isogeny.

\subsection{Fields of Definition for Isogenies}\label{Section3.1} Let $\mathcal{E}$ be a rational $\overline{\Q}$-isogeny class of non-CM elliptic curves. By definition, there exists $E_0\in\mathcal{E}$ with $j(E_0)\in \Q$, and for any $E\in \mathcal{E}$ there is a $\overline{\Q}$-isogeny $\varphi:E \rightarrow E_0$. Since the degree of closed points on $X_1(N)$ can be computed using any Weierstrass model of $E$ by Lemma \ref{lem:ResidueField}, we are free to replace $E$ and $E_0$ by quadratic twists in order to achieve a more convenient representation of $\varphi$. The lemma given below essentially follows from the fact that $\Q(j(E),j(E_0))=\Q(j(E))$ is contained in the residue field of any closed point on $X_1(N)$ associated to $E$, and this is the field of moduli of the isogeny $\varphi$; see \cite[$\S3.3$]{ClarkVolcanoes} or \cite[Corollary A.5]{CremonaNajmanQCurve}.

\begin{lem} \label{Cor:IsogenyDef}
Let $\mathcal{E}$ be a rational $\overline{\Q}$-isogeny class of non-CM elliptic curves and let $E \in \mathcal{E}$. Suppose $x=[E,P] \in X_1(\ell^k)$ for some prime number $\ell$ and positive integer $k$, and let $F \coloneqq \Q(x)$. There is a Weierstrass equation of $E/F$ for which $P \in E(F)$ and such that there exists an $F$-rational cyclic isogeny $\varphi: E \rightarrow E_0$ with $j(E_0)\in\Q$.
\end{lem}

\begin{proof}

Since $\mathcal{E}$ is rational, there exists $E_0\in \mathcal{E}$ with $j(E_0)\in \Q$. By definition there exists an isogeny $\varphi: E \rightarrow E_0$ defined over $\overline{\Q}$ which we may assume is cyclic of degree $N$; see Lemma A.1 in \cite{CremonaNajmanQCurve}. Let $C$ denote its kernel. Note that $F=\Q(j(E),\mathfrak{h}(P))$ by Lemma \ref{lem:ResidueField} and there exists a Weierstrass equation of $E/F$ with $P \in E(F)$. The proof strategy of  \cite[Proposition 3.2]{ClarkVolcanoes} shows $C$ is $F$-rational, as we will now explain. Suppose $\sigma(C) \neq C$ for some $\sigma \in \Gal_F$, and consider the induced isogeny $E^{\sigma} \rightarrow (E/C)^{\sigma}$. Since $j(E/C)=j(E_0) \in \Q$, we see that $j((E/C)^{\sigma})=j(E_0)$. Thus composition with an isomorphism to $E_0$ yields a cyclic $N$-isogeny $\psi: E \rightarrow E_0$ with kernel $\sigma(C)$. But having two cyclic $N$-isogenies from $E$ to $E_0$ with distinct kernels can happen only if $E$ has complex multiplication (see the last paragraph of the proof of \cite[Proposition 3.2]{ClarkVolcanoes} for details). We have reached a contradiction.
\end{proof}

\subsection{General divisibility conditions}\label{Section3.2} Let $\mathcal{E}$ be rational $\overline{\Q}$-isogeny class of non-CM elliptic curves. Fix a prime number $\ell$ and positive integer $k$. If $E \in \mathcal{E}$, then we can relate the degree of $[E,P] \in X_1(\ell^k)$ to the degree of $[E_0,P_0] \in X_1(\ell)$ for some $E_0\in\mathcal{E}$ having $j(E_0)\in \Q$. This is formalized in the following proposition, which strengthens \cite[Lemma 4.6]{OddDegQCurve} and proves Proposition \ref{Prop1.5intro}.

\begin{prop}\label{Prop1.6Restated}
Let $\mathcal{E}$ be a rational $\overline{\Q}$-isogeny class of non-CM elliptic curves. Suppose $\ell$ is a prime number and $k \in \Z^+$. There exists $E_0/\Q \in\mathcal{E}$ and $x\in X_1(\ell)$ with $j(x)=j(E_0)$ such that the degree of any point on $X_1(\ell^k)$ associated to $\mathcal{E}$ is divisible by 
\[\begin{cases}
\deg(x)\cdot \ell^{\max(0,2k-2-d)} \text{ if $\ell$ is odd},\\
\deg(x) \cdot \ell^{\max(0,2k-3-d)} \text{ if $\ell=2$},
\end{cases}\]
where $d \coloneqq \ord_{\ell}([\GL_2(\Z_{\ell}): \im \rho_{E_0, \ell^{\infty}}])$.
\end{prop}

\begin{remark}
Let $k \in\Z^{\geq 2}$. Since $\deg(X_1(\ell^k) \rightarrow X_1(\ell))=\ell^{2k-2}$ for $\ell$ odd and $\deg(X_1(2^k) \rightarrow X_1(2))=2^{2k-3}$, these lower bounds are best-possible whenever $d=0$. In this case $E_0$ itself is a minimal torsion curve for $X_1(\ell^k)$. This holds in certain cases; see, for example, Lemma \ref{MinTorsionCurveForIsogeny}. However, there are other $\overline{\Q}$-isogeny classes for which these bounds can be further refined. See Proposition \ref{DivConditionsPrime3}.
\end{remark}

\begin{proof}
Let $E \in \mathcal{E}$, and fix $P \in E$ of order $\ell^k$. Define $F \coloneqq \Q(j(E),\mathfrak{h}(P))$. By Lemma \ref{Cor:IsogenyDef}, there is a Weierstrass model of $E/F$ where $P \in E(F)$ and such that there exists an $F$-rational cyclic isogeny $\varphi: E \rightarrow E'$ with $j(E')\in\Q$. By \cite[Lemma 4.6]{OddDegQCurve}, we have $[F:\Q]$ is divisible by
\[\begin{cases}
\ell^{\max(0,2k-2-d)} \text{ if $\ell$ is odd},\\
\ell^{\max(0,2k-3-d)} \text{ if $\ell=2$},
\end{cases}\]
where $d=\ord_{\ell}([\GL_2(\Z_{\ell}): \im \rho_{E_0, \ell^{\infty}}])$ for any elliptic curve $E_0/\Q$ with $j(E_0)=j(E')$.

By \cite[Corollary 4.3 (2)]{OddDegQCurve}, the curve $E'$ has a point of order $\ell$ over an extension $F'/F$ of degree dividing $\ell$. In particular, there exists a closed point $x=[E',P'] \in X_1(\ell)$ such that $\Q(x) \subseteq F'$. Hence
\[
\deg(x) \mid \ell \cdot [F:\Q].
\]
Since $j(E')=j(E_0)$, there exists a point $P_0\in E_0$ such that the closed point $x=[E_0,P_0]$. If $\deg(x)$ is prime to $\ell$, then $\deg(x) \mid [F:\Q]$, and the conclusion follows since $F$ is the residue field of the closed point associated to $[E,P]$ by Lemma \ref{lem:ResidueField}. So suppose $\deg(x)=\ell \cdot n_0$. Note that $\ell \nmid n_0$, since $\deg(X_1(\ell) \rightarrow X_1(1))<\ell^2$, so $n_0 \mid [F:\Q]$. If there exists $x_1\in X_1(\ell)$ associated to $E_0$ with $\deg(x_1) \mid n_0$, then the conclusion follows with $x_1$ in place of $x$. So suppose not. By checking the possible images of the mod $\ell$ Galois representation associated to $E_0$ as in \cite[Theorem 5.6, Tables 1 \& 2]{GJNajman}, we see that we must be in one of the following cases (see Section 2.1 for an explanation of subgroup labels used):
\begin{itemize}
\item $\ell=5$, $\im \rho_{E_0,5}=5B.1.2, \,5B.1.3,\text{ or } 5B.4.2$,
\item $\ell=7$, $\im \rho_{E_0,7}=7B.1.3, \, 7B.1.4,\text{ or }7B.6.3$,
\item $\ell=13$, $\im \rho_{E_0,13}=13B.3.2, \, 13B.3.7,\, 13B.5.2, \text{ or }13B.4.2$,
\item $\ell=17$, $\im \rho_{E_0,17}=17B.4.6$,
\item $\ell=37$, $\im \rho_{E_0,37}=37B.8.2$.
\end{itemize}

In each case, there exists a $\Q$-rational cyclic subgroup $C$ of $E_0$ such that $E_0/C$ has mod $\ell$ image outside this list; see \cite[Theorem 3.32, Tables 3 \& 4]{sutherland}. That is, in each case there exists a point on $X_1(\ell)$ associated to $E_0/C$ of degree dividing $n_0$, and the result holds with $E_0/C$ in place of $E_0$.
\end{proof}

\begin{remark}\label{BadImageRmk}
From the proof, we see that the statement of Proposition \ref{Prop1.6Restated} holds for any $E_0/\Q \in \mathcal{E}$ which satisfies the following constraints:
\begin{itemize}
\item $\ell=5$, $\im \rho_{E_0,5} \neq 5B.1.2, \,5B.1.3,\text{ or } 5B.4.2$,
\item $\ell=7$, $\im \rho_{E_0,7} \neq 7B.1.3, \, 7B.1.4,\text{ or }7B.6.3$,
\item $\ell=13$, $\im \rho_{E_0,13} \neq 13B.3.2, \, 13B.3.7,\, 13B.5.2, \text{ or }13B.4.2$,
\item $\ell=17$, $\im \rho_{E_0,17} \neq 17B.4.6$,
\item $\ell=37$, $\im \rho_{E_0,37} \neq 37B.8.2$.
\end{itemize}
\end{remark}

\subsection{Image of Galois Representations Under Isogeny}\label{Section3.3}

\begin{prop} \label{LadicIsogenyLemma}
Let $E_1/F$ be a non-CM elliptic curve, and fix a prime number $\ell$. Suppose $\varphi: E_1 \rightarrow E_2$ is an $F$-rational cyclic $\ell^r$-isogeny of elliptic curves over $F$ for $r \in \Z^+$. Then:
\begin{enumerate}
\item $\im \rho_{E_2,\ell^k}$ is completely determined by $\im \rho_{E_1,\ell^{r+k}}$
\item If $\im \rho_{E_1,\ell^{\infty}}=\pi^{-1}(\im \rho_{E_1,\ell^{k_1}})$ for some $k_1 \in \Z^+$, then $\im \rho_{E_2,\ell^{\infty}}=\pi^{-1}(\im \rho_{E_2,\ell^{k_1+r}})$. Here, $\pi$ denotes the reduction map from $\GL_2(\Z_{\ell})$ to $\GL_2(\Z/\ell^{k_1}\Z)$ or $\GL_2(\Z/\ell^{k_1+r\Z})$, respectively.
\end{enumerate}
\end{prop}

\begin{proof}
Let $\{P, Q\}$ be a basis for $E_1[\ell^{r+k}]$, where $\ker(\varphi)=\langle \ell^k P \rangle$. Then with respect to this basis, for any $\sigma \in \Gal_{F}$, there exist $a, b, c ,d \in \Z$ such that

\[ \rho_{E_1,\ell^{r+k}}(\sigma)=
\begin{pmatrix}
a & b  \\
\ell^{r} c & d
\end{pmatrix}.
\]
One can check that $\{\varphi(P), \ell^{r} \varphi(Q)\}$ gives a basis for $E_2[\ell^k]$. Moreover, for $\sigma \in \Gal_{F}$, we have
\begin{align*}
\sigma(\varphi(P))&=\varphi(\sigma(P))=\varphi(aP+\ell^{r} c Q)=a\varphi(P)+c \ell^{r} \varphi(Q),\\
\sigma(\ell^{r} \varphi( Q))&=\ell^{r} \varphi(\sigma(Q))=\ell^{r}\varphi(bP+d Q)=\ell^{r}b\varphi(P)+d\ell^{r}\varphi( Q).
\end{align*}
Thus 
\[ \rho_{E_2,\ell^{k}}(\sigma)=
\begin{pmatrix}
a &\ell^{r} b  \\
c & d
\end{pmatrix},
\] and $\im \rho_{E_2,\ell^k}$ can be deduced from $\im \rho_{E_1,\ell^{r+k}}$.

Finally, suppose $\im \rho_{E_1,\ell^{\infty}}=\pi^{-1}(\im \rho_{E_1,\ell^{k_1}})$, and let $M \in \GL_2(\Z/\ell^{k_1+r+1}\Z)$ where 
\[
M \pmod{\ell^{k_1+r}} \in \im \rho_{E_2, \ell^{k_1+r}}.
\]
If we can show that $M \in \im \rho_{E_2, \ell^{k_1+r+1}}$, then $\im \rho_{E_2,\ell^{\infty}}=\pi^{-1}(\im \rho_{E_2,\ell^{k_1+r}})$ by, e.g., \cite[Proposition 3.5]{BELOV}. By the first paragraph, we may assume
\[
M=\begin{pmatrix}
a + \ell^{k_1+r}\alpha &\ell^{r} b+ \ell^{k_1+r}\beta  \\
c + \ell^{k_1+r}\gamma & d + \ell^{k_1+r}\delta
\end{pmatrix}=\begin{pmatrix}
a + \ell^{k_1+r}\alpha &\ell^{r}(b+ \ell^{k_1}\beta)  \\
c + \ell^{k_1+r}\gamma & d + \ell^{k_1+r}\delta
\end{pmatrix}
\]
for some 
\[
\begin{pmatrix}
a & b  \\
\ell^r c & d
\end{pmatrix} \in \im \rho_{E_1,\ell^{k_1+2r}}.
\]
Since $\im \rho_{E_1,\ell^{\infty}}=\pi^{-1}(\im \rho_{E_1,\ell^{k_1}})$, there exists $\sigma \in \Gal_F$ such that
\[
\rho_{E_1,\ell^{r+k_1+r+1}}(\sigma)=\begin{pmatrix}
a +\ell^{k_1+r}\alpha & b+\ell^{k_1}\beta  \\
\ell^r (c+\ell^{k_1+r}\gamma) & d+\ell^{k_1+r}\delta
\end{pmatrix},
\] as its reduction mod $\ell^{k_1}$ is in $\im \rho_{E_1,\ell^{k_1}}$. By the first paragraph,
\[
\rho_{E_2,\ell^{k_1+r+1}}(\sigma)=
\begin{pmatrix}
a +\ell^{k_1+r}\alpha & \ell^r(b+\ell^{k_1}\beta)  \\
c+\ell^{k_1+r}\gamma & d+\ell^{k_1+r}\delta
\end{pmatrix}=M. \qedhere
\] \end{proof}

\begin{cor}\label{ImageOnlyCor}
Let $E/\Q$ be a non-CM elliptic curve. For any prime number $\ell$ and $k \in \Z^+$, the degrees of closed points on $X_1(\ell^k)$ associated to elliptic curves $\ell^r$-isogenous to $E$ over $\overline{\Q}$ are entirely determined by $\im \rho_{E/\Q,\ell^{\infty}}$.
\end{cor}

\begin{proof}
Suppose there exists an isogeny $\varphi:E \rightarrow E'$ defined over $\overline{\Q}$ which is cyclic of order $\ell^r$. Thus there exists a point $P\in E$ of order $\ell^r$ such that $E' \cong E/\langle P \rangle$. Let $F\coloneqq \Q(\langle P \rangle)$, and let $\psi:E \rightarrow E/\langle P \rangle$ be the induced $F$-rational isogeny. By Proposition \ref{LadicIsogenyLemma} the $\ell$-adic Galois representation of $E/\langle P \rangle$, and thus the degrees of all closed points on $X_1(\ell^k)$ associated to  $E/\langle P \rangle$, is determined by $\im \rho_{E/F, \ell^{\infty}}$. Since $\im \rho_{E/F, \ell^{\infty}}$ can be obtained from $\im \rho_{E/\Q, \ell^{\infty}}$ and the definition of $F$, the result follows.
\end{proof}

\section{Properties of Minimal Torsion Curves}

Let $\mathcal{E}$ be a rational $\overline{\Q}$-isogeny class of elliptic curves. In this section, we show that for a fixed positive integer $N$, there are finitely many minimal torsion curves in $\mathcal{E}$ for $X_1(N)$. In some cases, there is a unique minimal torsion curve, but this should not be expected in general (see Remark \ref{RemarkUniqueness}). Next, we suppose $\mathcal{E}$ is non-CM,\footnote{Recall $\text{End}(E) \otimes \Q$ is an isogeny invariant, so the entire class $\mathcal{E}$ either is CM or non-CM.} and let $E_0\in \mathcal{E}$ be an elliptic curve with rational $j$-invariant. Fix a prime $\ell$. In Theorem \ref{PropMinTorsionCurveBound} we show that for any $k \in \Z^+$, there exists a minimal torsion curve $E\in \mathcal{E}$ for $X_1(\ell^k)$ such that the degree of the isogeny from $E$ to $E_0$ is at most $C$, where $C$ is a constant that does not depend on $k$. This implies Theorem \ref{IntroPropSerre}, as stated in the introduction.

\subsection{Minimal Torsion Curves for Fixed Modular Curve} \begin{prop}\label{Prop:FiniteMinCurve}
Let $\mathcal{E}$ be a rational $\overline{\Q}$-isogeny class of elliptic curves. For a fixed positive integer $N$, there exist finitely many minimal torsion curves for $X_1(N)$ up to isomorphism over $\overline{\Q}$.
\end{prop}

\begin{proof}
Let $d$ be the minimal degree of a point on $X_1(N)$ associated to $\mathcal{E}$, and let $j_{min}$ be the $j$-invariant of a minimal torsion curve for $\mathcal{E}$. Then
$
[\Q(j_{min}):\Q] \leq d.
$ If $\mathcal{E}$ is CM, then the conclusion follows as there are only finitely many CM $j$-invariants in an extension of bounded degree (since there are only finitely many imaginary quadratic fields---and hence imaginary quadratic orders---of a given class number \cite{heilbronn}, this is a consequence of \cite[Theorem 11.1 \& Proposition 13.2]{cox}). So suppose $\mathcal{E}$ is non-CM.

Let $E_{min}\in \mathcal{E}$ be an elliptic curve with $j(E_{min})=j_{min}$. Since $\mathcal{E}$ is rational, there exists an elliptic curve $E_0/\Q \in \mathcal{E}$ with an isogeny 
$\varphi: E_0 \rightarrow E_{min}$ defined over $\overline{\Q}$. We may assume $\varphi$ is cyclic of degree $n$ by \cite[Lemma A.1]{CremonaNajmanQCurve}. By \cite[Proposition 3.2]{ClarkVolcanoes} or \cite[Corollary A.5]{CremonaNajmanQCurve}, the field of moduli of this isogeny is 
\[
\Q(j_{min},j(E_0))=\Q(j_{min}).
\]
Thus $\varphi$ can be defined over $\Q(j_{min})$ by replacing $E_0$ and $E_{min}$ with appropriate twists, and in particular a twist of $E_0$ defined over this field has a rational cyclic $n$-isogeny. Since isogenies are twist invariant (i.e., if $E_1, E_2$ are two elliptic curves over $F$ with the same $j$-invariant, then $E_1$ has an $F$-rational cyclic $n$-isogeny if and only if $E_2$ does), it follows that $E_0$ attains a rational cyclic $n$-isogeny over $\Q(j_{min})$, a number field of degree at most $d$.

By Serre's Open Image Theorem \cite{serre72}, the image of the adelic Galois representation associated to $E_0/\Q$ has finite index in $\GL_2(\widehat{\Z})$. For $d$ fixed, it follows that there is a bound on the order of a cyclic subgroup of $E_0$ which can become $F$-rational over \emph{any} number field $F$ of degree at most $d$. Thus there are only finitely many choices for $C$ such that $E_{min} \cong E_0/C$ over $\overline{\Q}$.
\end{proof}

\begin{remark}\label{RemarkUniqueness}
The minimal torsion curve within a fixed geometric isogeny class may or may not be unique (up to isomorphism over $\overline{\Q}$). For example, let $E_0/\Q$ be the elliptic curve with LMFDB label \href{https://www.lmfdb.org/EllipticCurve/Q/38/b/2}{38.b2} and let $\mathcal{E}$ denote its geometric isogeny class. Since there exists $P_0 \in E_0(\Q)$ of order 5, the least degree of a point on $X_1(5)$ associated to $\mathcal{E}$ is 1, and so any minimal torsion curve for $X_1(5)$ must have $j$-invariant in $\Q$. By Lemma \ref{Cor:IsogenyDef}, the only other elliptic curve $E' \in \mathcal{E}$ with $j$-invariant in $\Q$ has $j(E')=-\frac{37966934881}{4952198}$. However, this curve does not give points on $X_1(5)$ of minimal degree. Thus $E_0$ is a unique minimal torsion curve for $X_1(5)$, up to $\overline{\Q}$-isomorphism.

On the other hand, let $E_0/\Q$ be the elliptic curve with LMFDB label \href{https://www.lmfdb.org/EllipticCurve/Q/50/b/1}{50.b1} and let $\mathcal{E}$ denote its geometric isogeny class. Then a similar argument shows that \emph{any} elliptic curve $\Q$-isogenous to $E_0$ is a minimal torsion curve for $X_1(3)$, giving 4 distinct minimal torsion curves for this class up to $\overline{\Q}$-isomorphism. Representatives for these curves can be found in the LMFDB isogeny class \href{https://www.lmfdb.org/EllipticCurve/Q/50/b/}{50.b}. 
\end{remark}

\subsection{Minimal Torsion Curves of Varying Level}

If we allow $N$ to vary, there may be infinitely many elliptic curves (up to isomorphism over $\overline{\Q}$) within a fixed geometric isogeny class which are minimal for $X_1(N)$. For example, suppose the $\ell$-adic Galois representation of a non-CM elliptic curve $E/\Q$ is surjective, and let $C$ be a cyclic subgroup of $E$ of order $\ell^k$ for $k \in \Z^+$. The elliptic curve $E/C$ can be defined over the extension $\Q(C)$ of degree $\ell^{k-1}(\ell+1)$ and possesses a $\Q(C)$-rational cyclic $\ell^k$-isogeny. By Proposition \ref{IsogenyToRationalPoints}, the curve $E/C$ gives a closed point on $X_1(\ell^k)$ of degree
$
\ell^{2k-2}(\ell^2-1)/2.
$ This is minimal for the geometric isogeny class of $E$ by Proposition \ref{Prop1.6Restated}. However, if $\mathcal{E}$ is a rational $\overline{\Q}$-isogeny class of non-CM elliptic curves, there always exists a minimal torsion curve for $X_1(\ell^k)$ whose isogeny degree to a rational elliptic curve is bounded.

\begin{thm}\label{IntroPropSerre}\label{PropMinTorsionCurveBound}
Let $\mathcal{E}$ be a rational $\overline{\Q}$-isogeny class of non-CM elliptic curves, and let $\ell$ be prime. There exists a constant $C=C(\mathcal{E},\ell)$ such for any $k \in \Z^+$ a point of least degree on $X_1(\ell^k)$ coming from $\mathcal{E}$ can be associated to $j_{min}\in \mathcal{E}$ which is $d$-isogenous to a rational $j$-invariant for some $d\leq C$.
\end{thm}

\begin{proof}
By Proposition \ref{Prop1.6Restated}, there exists $E_0/\Q\in \mathcal{E}$ and $x=[E_0,P_0] \in X_1(\ell)$ such that the degree of any point on $X_1(\ell^n)$ associated to $\mathcal{E}$ is divisible by 
\[d_{\text{min}}(\ell^n)\coloneqq \begin{cases}
\deg(x)\cdot \ell^{\max(0,2n-2-d)} \text{ if $\ell$ is odd},\\
\deg(x) \cdot \ell^{\max(0,2n-3-d)} \text{ if $\ell=2$},
\end{cases}\]
where $d \coloneqq \ord_{\ell}([\GL_2(\Z_{\ell}): \im \rho_{E_0, \ell^{\infty}}])$. 

We note $\im \rho_{E_0,\ell^{\infty}}$ has level $\ell^{k_0}$ for some $k_0 \in \Z^{\geq 0}$ by Serre's Open Image Theorem \cite{serre72}. Replacing $k_0$ with a larger integer if necessary, we may assume $2k_0-2-d \geq0$ if $\ell$ is odd and $2k_0-3-d \geq 0$ if $\ell=2$. Now, let $k$ be an integer with $k \geq k_0$. Then by Proposition \ref{Prop1.6Restated}, there exists $\alpha_0 \in \Z^+$ such that the least degree of a closed point on $X_1(\ell^k)$ associated to $E_0 \in \mathcal{E}$ is
\begin{align*}
d_{\text{min},E_0}(\ell^k)&= d_{\text{min}}(\ell^k) \cdot \alpha_0\\
&= d_{\text{min}}(\ell^{k_0})\cdot \ell^{2(k-k_0)}\cdot \alpha_0.
\end{align*}
Since $\deg(X_1(\ell^k) \rightarrow X_1(\ell^{k_0}))=\ell^{2(k-k_0)}$, the assumptions concerning the level of $\im \rho_{E_0,\ell^{\infty}}$ imply  
\[
d_{\text{min},E_0}(\ell^{k_0})= d_{\text{min}}(\ell^{k_0}) \cdot \alpha_0.
\]
In particular, $\alpha_0$ does not depend on $k$.

 Suppose first that $E_0$ is a minimal torsion curve for $X_1(\ell^k)$ for sufficiently large $k$. Then by Proposition \ref{Prop:FiniteMinCurve} there exists a finite collection of elliptic curves $E_0, E_1, \dots, E_s$ such that for any positive integer $n$ a point of least degree on $X_1(\ell^n)$ coming from $\mathcal{E}$ can be associated to $E_i$ for some $0 \leq i \leq s$. We may assume the isogeny $\varphi:E_i \rightarrow E_0$ is cyclic of degree $d_i$ by \cite[Lemma A.1]{CremonaNajmanQCurve}. The result follows in this case with $C=\max\{d_0, d_1, \cdots, d_s\}$.
 
On the other hand, suppose $E_0$ is not a minimal torsion curve for $X_1(\ell^k)$ for all $k\geq k_0$. Then there exists $E_1 \in \mathcal{E}$ which is a minimal torsion curve for some $X_1(\ell^{k_1})$ where $k_1 \geq k_0$. Choose a Weierstrass equation for $E_1$ defined over $F\coloneqq \Q(j(E_1))$. Replacing $k_1$ with a larger integer if necessary, we may assume that $\im \rho_{E_1,\ell^{\infty}}=\pi^{-1}(\im \rho_{E_1, \ell^{k_1}})$. By assumption, there exists a positive integer $\alpha_1$ with $\alpha_1<\alpha_0$ such that 
\[
d_{\text{min},E_1}(\ell^{k})=d_{\text{min}}(\ell^k)\cdot \alpha_1
\]
for all $k \geq k_1$. If $E_1$ is a minimal torsion curve for $X_1(\ell^k)$ for sufficiently large $k$, then we are done as before. Continuing in this way produces a decreasing sequence of positive integers $\alpha_{0}>\alpha_1>\alpha_2 \dots$, so the process must stop after a finite number of steps.
\end{proof}

\begin{remark}\label{RemarkExplicitC}
It would be interesting to make the constant $C=C(\mathcal{E},\ell)$ explicit, perhaps in terms of invariants one could compute from the elliptic curve $E_0/\Q$.
\end{remark}

From Theorem \ref{PropMinTorsionCurveBound} we can immediately deduce the following corollary.

\begin{cor}
Let $\mathcal{E}$ be a rational $\overline{\Q}$-isogeny class of non-CM elliptic curves. There exists a finite collection of $j$-invariants $j_1, \dots, j_s \in \mathcal{E}$ such that for any $n \in \Z^+$, a point of least degree on $X_1(\ell^n)$ coming from $\mathcal{E}$ can be associated to an elliptic curve with $j$-invariant $j_i$ for some $1 \leq i \leq s$.
\end{cor}

\section{Results for non-CM classes and primes $\ell \geq 5$}
The main result of this section is the following, which proves Theorem \ref{Thm1.3} for $\ell \geq 5$. It refines \cite[Proposition 4.1]{OddDegQCurve}, which gives divisibility conditions which hold across all rational non-CM classes $\mathcal{E}$ and does not address whether they are best-possible.
\begin{prop} \label{BestDivLargePrimes}
Let $\mathcal{E}$ be a rational $\overline{\Q}$-isogeny class of non-CM elliptic curves. Suppose $\ell \geq 5$ is prime. If $E \in \mathcal{E}$ and $x=[E,P]\in X_1(\ell^k)$ is a point of odd degree for $k \in \Z^+$, then $\ell \in \{5,7,11,13\}$ and $\delta \mid \deg(x)$ for $\delta$ defined as follows:
\begin{itemize}
\item If $\ell=5$ and $\mathcal{E}$ does not contain $E'/\Q$ with a rational cyclic 25-isogeny, then $\delta=5^{2k-2}$.
\item If $\ell=5$ and $\mathcal{E}$ contains $E'/\Q$ with a rational cyclic 25-isogeny, then $\delta=5^{\max(0,2k-3)}$.
\item If $\ell=7$ and $\mathcal{E}$ contains $E'/\Q$ with $\im \rho_{E',7}\in\{7B.1.1,7B.1.6,7B.6.1\}$, then $\delta=7^{2k-2}$.
\item If $\ell=7$ and $\mathcal{E}$ contains $E'/\Q$ with $\im \rho_{E',7}\in\{7B.1.2,7B.6.2, 7B.2.1,7B\}$, then $\delta=3\cdot 7^{2k-2}$.
\item If $\ell=7$ and $\mathcal{E}$ contains $E'$ with $j(E')= 3^3\cdot5\cdot7^5/2^7$, then $\delta=9 \cdot 7^{\max(0,2k-3)}$.
\item If $\ell =11$, then $\delta=5 \cdot 11^{2k-2}$.
\item If $\ell =13$, then $\delta=3 \cdot 13^{2k-2}$.
\end{itemize}
Moreover, there exists a point of degree $\delta$ on $X_1(\ell^k)$ associated to $j_{min} \in \mathcal{E}$ which is at most $\ell$-isogenous to a rational $j$-invariant. One can take $j_{min}\in \Q$ unless $\ell=7$ and $\mathcal{E}$ contains the elliptic curve with $j$-invariant  $3^3\cdot5\cdot7^5/2^7$.
\end{prop}

\subsection{A Preliminary Result} We begin with a preliminary result concerning geometric isogeny classes $\mathcal{E}$ containing an elliptic curve over $\Q$ with a rational cyclic isogeny. When this occurs, we will say \textbf{$\mathcal{E}$ gives a point} in $X_0(\ell)(\Q)$. In this case, the result is largely a consequence of the following theorem and Proposition \ref{Prop1.6Restated}.
\begin{thm}[Greenberg \cite{greenberg2012}, Greenberg, Rubin, Silverberg, and Stoll \cite{greenberg2014}]\label{GreenbergThm}
Let $E/\Q$ be a non-CM elliptic curve with a $\Q$-rational cyclic isogeny of prime degree $\ell$. By choice of basis, we may assume  $\im \rho_{E,\ell^{\infty}}$ is a subgroup of $\GL_2(\Z_{\ell})$.
\begin{enumerate}
\item If $\ell\geq 7$, then $\im \rho_{E,\ell^{\infty}}$ contains a Sylow pro-$\ell$ subgroup of $\GL_2(\Z_{\ell})$.
\item Let $\ell=5$. If none of the elliptic curves in the $\Q$-isogeny class of $E$ has two independent isogenies of degree 5, then $\im \rho_{E,5^{\infty}}$ contains a Sylow pro-$5$ subgroup of $\GL_2(\Z_5)$. Otherwise the index of $\im \rho_{E,5^{\infty}}$ in $\GL_2(\Z_5)$ is divisible by 5, but not by 25.
\end{enumerate}
\end{thm}
\begin{proof}
If $\ell \geq 11$ or $\ell=5$, this follows from \cite[Theorems 1 \& 2]{greenberg2012}. Note if $\ell \geq 11$ or if $\ell=5$, the assumption in \cite[Theorems 1]{greenberg2012} holds, as explained in \cite[Remark 4.2.1, p. 1186--1187]{greenberg2012}. For $\ell=7$, this follows from \cite[Theorem 5.5]{greenberg2014}.
\end{proof}
\begin{lem} \label{MinTorsionCurveForIsogeny}
Suppose $\ell\geq 5$ is prime and $k \in \Z^+$. Let $\mathcal{E}$ be a rational $\overline{\Q}$-isogeny class of non-CM elliptic curves which gives a point in $X_0(\ell)(\Q)$. Then there exists $E_0/\Q \in\mathcal{E}$ and $x\in  X_1(\ell)$ with $j(x)=j(E_0)$ such that the degree of any point on $X_1(\ell^k)$ arising from $\mathcal{E}$ is divisible by 
\[\delta \coloneqq \begin{cases}
\deg(x)\cdot \ell^{2k-2} \text{ if $\ell \neq 5$ or $\mathcal{E}$ does not contain $E'/\Q$ with a rational cyclic 25-isogeny},\\
\deg(x)\cdot 5^{\max(0,2k-3)} \text{ if $\ell=5$ and $\mathcal{E}$ contains $E'/\Q$ with a rational cyclic 25-isogeny}.
\end{cases}\]
Moreover, there exists a point of degree $\delta$ on $X_1(\ell^k)$ associated to $j_{min} \in \mathcal{E}$  with $j_{min} \in \Q$.
\end{lem}

\begin{proof}
Suppose first that $\ell>5$ or {that} $\ell=5$ and $\mathcal{E}$ does not contain $E'/\Q$ with a rational cyclic 25-isogeny. By Proposition \ref{Prop1.6Restated}, there exists $E_0/\Q\in \mathcal{E}$ and $x=[E_0,P_0]\in X_1(\ell)$ such that 
\[
\deg(x)\cdot \ell^{\max(0,2k-2-d)}
\] divides the degree of any point on $X_1(\ell^k)$ associated to $\mathcal{E}$, where $d=\ord_{\ell}([\GL_2(\Z_{\ell}): \im \rho_{E_0, \ell^{\infty}}])$.
By Theorem \ref{GreenbergThm}, we have $d=0$. Since $\deg(X_1(\ell^k) \rightarrow X_1(\ell))=\ell^{2k-2}$, lifts of $x$ on $X_1(\ell^k)$ show this divisibility condition is best-possible with $j_{min}=j(E_0)$.

Now, suppose $\ell=5$ and $\mathcal{E}$ contains $E'/\Q$ with a rational cyclic 25-isogeny. Then by Theorem \ref{GreenbergThm}, we have $ \ord_{5}([\GL_2(\Z_{5}): \im \rho_{E', 5^{\infty}}])=1$. By replacing $E'$ with a curve $\Q$-isogenous, we may assume $E'$ has two independent 5-isogenies. The index of the image of the 5-adic Galois representation is unchanged by \cite[Proposition 2.1.1]{greenberg2012}. Then $\im \rho_{E',5}$ is one of the following, and we consider each case separately:
\begin{itemize}
\item 5Cs.1.1, 5Cs.1.3, or 5Cs.4.1: Replacing $E'$ with a quadratic twist if necessary, we may assume $\im \rho_{E',5}=5Cs.1.1$. By Proposition \ref{Prop1.6Restated} and Remark \ref{BadImageRmk}, we see $5^{\max(0,2k-3)}$ divides the degree of any point on $X_1(5^k)$ associated to $\mathcal{E}$. The curve $E'$ has a subgroup $C_1$ generated by a rational point $P$ of order 5 and an independent rational cyclic subgroup $C_2$ of order 5. Then $E_2\coloneqq E'/C_2$ has a rational cyclic 25-isogeny, and the image of $P$ is a point of order 5 in $E_2(\Q)$ lying in its kernel. Thus the image of the isogeny character $\chi: \Gal_{\Q} \rightarrow (\Z/25\Z)^{\times}$ lands in the subgroup $\{a : a \equiv 1 \pmod{5}\}$ and so has order dividing 5. It follows that  $E_2$ attains a rational point of order 25 in an extension of degree dividing 5. By Proposition \ref{prop:Degree}, there are lifts of this point on $X_1(5^k)$ of degree at most $5^{2k-3}$ for all $k \geq 2$. Thus the condition is best-possible for all $k$ with $j_{min}=j(E_2)$.
\item 5Cs: By Proposition \ref{Prop1.6Restated} and Remark \ref{BadImageRmk}, we see $2\cdot 5^{2k-3}$ divides the degree of any point on $X_1(5^k)$ associated to $\mathcal{E}$. Note $E'$ has two independent 5-isogenies with kernels $C_1$ and $C_2$. Then $E_2\coloneqq E'/C_2$ has a rational cyclic 25-isogeny. By Proposition \ref{IsogenyToRationalPoints}, the curve $E_2$ gives a closed point on $X_1(5)$ in degree dividing 2 and a closed point on $X_1(25)$ in degree dividing 10. Again by considering lifts of this point with degree bounds given by Proposition \ref{prop:Degree}, we see the divisibility condition is best-possible for all $k$ with $j_{min}=j(E_2)$. \qedhere
\end{itemize}
\end{proof}

\subsection{Proof of Proposition \ref{BestDivLargePrimes}}
Let $F=\Q(x)$. By Lemma \ref{Cor:IsogenyDef}, there is a model of $E/F$ where $P \in E(F)$ and such that there exists an $F$-rational cyclic isogeny $\varphi:E \rightarrow E'$ with $j(E')\in\Q$. Since $E$ has an $\ell$-isogeny over $F$, so does $E'$ by \cite[Proposition 3.2]{CremonaNajmanQCurve}. It follows from \cite[Proposition 3.3]{CremonaNajmanQCurve} that any $E_0/\Q\in\mathcal{E}$ with $j(E_0)=j(E')$ has a $\Q$-rational cyclic $\ell$-isogeny, unless $\ell=7$ and $j(E_0)=3^3\cdot5\cdot7^5/2^7$. Work of Mazur \cite{mazur78} implies $\ell \leq 37$. By applying Lemma \ref{MinTorsionCurveForIsogeny} and checking the possible $x\in X_1(\ell)$ of odd degree associated to elliptic curves over $\Q$ in $\mathcal{E}$, where we may omit images as appearing in Remark \ref{BadImageRmk}, we see that we are done unless $\ell=7$ and $\mathcal{E}$ contains an elliptic curve with $j$-invariant $3^3 \cdot 5 \cdot 7^5 / 2^7$.

Let us discuss the case $\ell=7$ and $j(E_0)=3^3\cdot5\cdot7^5/2^7$. Recall $P \in E(F)$ is a point of order $7^k$. Notice that $7^{k-1} P$ is a point of order $7$ on $E$ and is also defined over $F$ on $E$. By \cite[Corollary 4.3]{OddDegQCurve}, the curve $E'$ has a rational point of order $7$ over an extension $F'/F$ of degree $1$ or $7$. As $j(E') = 3^3 \cdot 5 \cdot 7^5 / 2^7$, a computation with division polynomials shows the residue field of a closed point on $X_1(7)$ associated to $E'$ has degree 6 or 9, so $[F':\Q]$ is divisible by 6 or 9. Since $[F':F]$ divides 7, it follows that 6 or 9 must divide $[F:\Q]$. Since $F$ is an extension of odd degree, we must have $9 \mid [F:\Q]$. Moreover, in \cite[Proposition 4.1]{OddDegQCurve}, it is proven that $3 \cdot 7^{\text{max} (0, 2k-3)}$ divides $[F: \Q]$. Therefore, $9 \cdot 7^{\text{max} (0, 2k-3)}$ divides $[F:\Q]$.

Conversely, we will now show there is a point on $X_1(7^k)$ of degree $9\cdot7^{\text{max}(0, 2k-3)}$ which is associated to $\mathcal{E}$. By replacing $E_0$ with a quadratic twist if necessary, we may assume $E_0$ has LMFDB label \href{https://www.lmfdb.org/EllipticCurve/Q/2450/y/1}{2450.y1}. A computation with division polynomials confirms that $E_0$ gives a closed point on $X_1(7)$ of degree 9, which fulfills the $k=0$ case. In addition, the mod 7 image of $E_0$ is 7Ns.2.1, which is of order 18 and generated by the following matrices:
\[\begin{pmatrix} 0 &1 \\ 1 & 0 \end{pmatrix}, \begin{pmatrix}
2 & 0\\0 &1
\end{pmatrix}
\]

A \href{https://github.com/abbey-bourdon/minimal_torsion_curves/blob/main/Proposition5.1/code}{Magma computation} shows that 7Ns.2.1 contains an index 3 subgroup conjugate to the group generated by
\[\begin{pmatrix} 1& 0 \\ 0 & 6 \end{pmatrix}, \begin{pmatrix}
2 & 0 \\ 0 & 2
\end{pmatrix}.
\]
Thus over a cubic extension $F$, the curve $E_0$ attains two independent 7 isogenies, and is $F$-isogenous to an elliptic curve $E_1/F$ with a rational cyclic 49-isogeny. By Proposition \ref{IsogenyToRationalPoints}, the curve $E_1$ gives a closed point on $X_1(49)$ of degree dividing $9 \cdot 7$. The previous paragraph shows it must have degree exactly $9 \cdot 7$, and since $\deg(X_1(7^k) \rightarrow X_1(7^2))$ has degree $7^{2k-4}$, the divisibility conditions of Proposition \ref{Prop1.6Restated} are the best-possible.

\section{Results for non-CM classes and $\ell  =3$}
In this section, we will prove the following result, which includes instances where the divisibility conditions of Proposition \ref{Prop1.6Restated} can be improved. It proves Theorem \ref{Thm1.3} if $\ell=3$. See Section 2.1 for a discussion of the notation used for subgroups of $\GL_2(\Z_3)$.
\begin{prop} \label{DivConditionsPrime3}
Let $\mathcal{E}$ be a rational $\overline{\Q}$-isogeny class of non-CM elliptic curves. If $E \in \mathcal{E}$ and $x=[E,P]\in X_1(3^k)$ is a point of odd degree for $k \in \Z^+$, then $\delta \mid \deg(x)$ for $\delta$ defined as follows:
\[\delta \coloneqq \begin{cases}
3^{\max(0,2k-3)} \text{ if there is $E'/\Q \in \mathcal{E}$ with $\im \rho_{E',3^{\infty}} \in \{9.36.0.6, 9.36.0.8\}$},\\
3^{\max(0,2k-2-d)} \text{ otherwise},
\end{cases}\]
where $d =\ord_{3}([\GL_2(\Z_{3}): \im \rho_{E_0, 3^{\infty}}])$ for any $E_0/\Q \in \mathcal{E}$. These are generally best-possible:
\begin{enumerate}
\item Suppose there is no $E'/\Q \in \mathcal{E}$ with $\im \rho_{E',3^{\infty}} \in \{9.12.0.2, 9.36.0.2, 9.36.0.7, 9.36.0.8\}$. For any $k \in \Z^+$, there exists a point of degree $\delta$ on $X_1(3^k)$ associated to $j_{min} \in \mathcal{E}$.
\item Suppose there is an $E'/\Q \in \mathcal{E}$ with $\im \rho_{E',3^{\infty}} \in \{9.12.0.2, 9.36.0.7, 9.36.0.8\}$. For $k=1$ or $k \geq 3$, there exists a point of degree $\delta$ on $X_1(3^k)$ associated to $j_{\min} \in \mathcal{E}$. If $k=2$, then $3\delta \mid \deg(x)$ and there exists a point of degree $3\delta$ on $X_1(3^k)$ associated to $j_{min} \in \mathcal{E}$. 
\item Suppose there is an $E'/\Q \in \mathcal{E}$ with $\im \rho_{E',3^{\infty}}=9.36.0.2$. For $k=1$ or $k \geq 4$, there exists a point of degree $\delta$ on $X_1(3^k)$ associated to $j_{min} \in \mathcal{E}$. Otherwise $3\delta \mid \deg(x)$ and there exists a point of degree $3\delta$ on $X_1(3^k)$ associated to $j_{min} \in \mathcal{E}$. 
\end{enumerate}
One can take $j_{min} \in \Q$, unless we are in case (2) with $k>2$ or case (3) with $k>3$; in these latter cases one can take $j_{min}$ to be $3$-isogenous to a rational $j$-invariant.
\end{prop}

The proof shows that $d \leq 2$ for the classes which produce points of odd degree. Thus we immediately deduce the following corollary.

\begin{cor}
Let $\mathcal{E}$ be a rational $\overline{\Q}$-isogeny class of non-CM elliptic curves. If $E \in \mathcal{E}$ and $x=[E,P]\in X_1(3^k)$ is a point of odd degree, then $3^{\max(0,2k-4)} \mid  \deg(x)$, and this is the best possible across all such $\mathcal{E}$.
\end{cor}

\begin{remark} \label{Remark5.4}
Suppose there exists $E' \in \mathcal{E}$ with $\im \rho_{E',3^{\infty}} =9.36.0.6$. By Proposition \ref{DivConditionsPrime3}, any odd degree point on $X_1(3^k)$ associated to an elliptic curve in $\mathcal{E}$ has degree divisible by $3^{\max(0,2k-3)}$. By Proposition \ref{Prop1.6Restated}, any point of even degree must be divisible by $2 \cdot  3^{\max(0,2k-4)}$. Since there exists $x \in X_1(3)$ of degree 1 associated to $E'$, this strengthens the lower bound in Proposition \ref{Prop1.6Restated} by a factor of 2 or 3, respectively. \end{remark}

\subsection{Preliminary Results} The goal of this section is to prove Proposition \ref{Prop5.7}, which obtains the divisibility condition of Proposition \ref{DivConditionsPrime3} in the case where $\mathcal{E}$ contains $E'/\Q$ with $\im \rho_{E', 3^{\infty}}=9.36.0.6$ or $9.36.0.8$. We start by proving two lemmas.
\begin{lem} \label{Lem3.8} %Lem3.9 
Suppose $F$ is a number field of odd degree and $E/F$ is a non-CM elliptic curve with $P\in E(F)$ of order $3^k$ for $k \in \Z^+$. Let $\varphi:E \rightarrow E'$ be an $F$-rational isogeny, where there exists $E_0/\Q$ of with $j(E_0)=j(E')$ and $d \coloneqq \ord_{3}([\GL_2(\Z_{3}): \im \rho_{E_0, 3^{\infty}}])$. If $3^{\max(0,2k-1-d)} \nmid [F:\Q]$, then with respect to the basis $\{P,Q\}$ of $E[3^k]$ we have
\[
\im \rho_{E/F,3^k}=\left\{\begin{pmatrix}
1 & x \\
0 & y
\end{pmatrix}|\, x \in \Z/3^k\Z, y \in (\Z/3^k\Z)^{\times} \right\} 
\]
and $\im \rho_{E/F,3^{\infty}} = \pi^{-1}(\im \rho_{E/F,3^k})$. 
\end{lem}

\begin{proof}
Suppose $3^{\max(0,2k-1-d)} \nmid [F:\Q]$, so in particular $2k-1-d>0$. Let $\{P,Q\}$ be a basis of $E[3^k]$. Replacing $F$ with at worst a quadratic extension $L/F$, we may view $\varphi$ as an $L$-isogeny from $E$ to $E_0/L$.  Then $\im \rho_{E/L,3^{k}}$  is contained in
\[
H \coloneqq \left\{\begin{pmatrix}
1 & x \\
0 & y
\end{pmatrix}|\, x \in \Z/3^k\Z, y \in (\Z/3^k\Z)^{\times} \right\}, 
\] which has order $3^k \cdot \varphi(3^k)=3^{2k-1}\cdot 2$. If $\ord_{3}(\#\im \rho_{E/L,{3}^{k}})<{2k-1}$, then the index of the mod $3^k$ Galois representation of $E/L$ is divisible by $3^{2k-1}$. Thus
\[
3^{2k-1} \mid [\GL_2(\Z_{3}):\im \rho_{E/L,3^{\infty}}].\]
By \cite[ Lemma 4.5]{OddDegQCurve}, we have $3^{2k-1} \mid [\GL_2(\Z_{3}): \im \rho_{E_0/\Q, 3^{\infty}}] \cdot[L \cap \Q(E_0[3^{\infty}]):\Q]$. Since
\[
\ord_{3}([\GL_2(\Z_{3}): \im \rho_{E_0/\Q, 3^{\infty}}])=d,
\]
it follows that $3^{2k-1-d} \mid [L \cap \Q(E_0[3^{\infty}]):\Q]$. Since $L$ is at {most} a quadratic extension of $F$, then $3^{2k-1-d} \mid [F:\Q]$, contradicting our assumption. So we may assume $\ord_{3}(\#\im \rho_{E/L,3^{k}})={2k-1}$. 

Note $\im \rho_{E/F,3^{k}}$ contained in $H$ as well, and since $L/F$ is at worst a quadratic extension, we have $\ord_{3}(\#\im \rho_{E/F,3^{k}})={2k-1}$. If $\im \rho_{E/F,3^{k}}$ is properly contained in $H$, then $\#\im \rho_{E/F,3^{k}}=3^{2k-1}$. Since $\Q(\zeta_{3^k}) \subseteq F(E[3^k])$, it follows that 2 must divide $[F:\Q]$. This contradicts $F$ having odd degree. Hence $\im \rho_{E/F,3^{k}}=H$. 

If $\im \rho_{E/F,3^{\infty}} \neq \pi^{-1}(\im \rho_{E/F,3^k})$, then $[F(E[3^{k+1}]):F(E[3^k])]$ divides $3^3$; see, for example, \cite[Proposition 3.5]{BELOV}. Since $\# \Gal(F(E[3^k])/F) =3^{2k-1}\cdot 2$, we have
\[
\# \Gal(F(E[3^{k+1}])/F) \mid 3^3\cdot 3^{2k-1}\cdot 2=3^{2k+2}\cdot 2.
\]
It follows that $\# \Gal(L(E[3^{k+1}])/L) \mid 3^{2k+2}\cdot 2$, and so the index of the 3-adic Galois representation of $E/L$ is divisible by at least $3^{2k-1}\cdot 8$. By \cite[Lemma 4.5]{OddDegQCurve}, we have $3^{2k-1}\cdot 8 \mid [\GL_2(\Z_{3}): \im \rho_{E_0/\Q, 3^{\infty}}] \cdot[L \cap \Q(E_0[3^{\infty}]):\Q]$. It follows that $3^{2k-1-d} \mid [L \cap \Q(E_0[3^{\infty}]):\Q]$. Since $L$ is at worst a quadratic extension of $F$, then $3^{2k-1-d} \mid [F:\Q]$, contradicting our assumption. Thus, $\im \rho_{E/F,3^{\infty}} = \pi^{-1}(\im \rho_{E/F,3^k})$.
\end{proof}

\begin{lem} \label{Lem3.10}
Suppose $F$ is a number field of odd degree and $E/F$ is a non-CM elliptic curve with $P\in E(F)$ of order $3^k$, $k \geq 2$. Let $\varphi:E \rightarrow E'$ be an $F$-rational isogeny of degree $3^r$ for $r \in \Z^+$, where there exists $E_0/\Q$ with $j(E_0)=j(E')$ and $d \coloneqq \ord_{3}([\GL_2(\Z_{3}): \im \rho_{E_0, 3^{\infty}}])$. If $3^{\max(0,2k-1-d)} \nmid [F:\Q]$, then $r \leq k$ and $\ker(\varphi) \subseteq \langle P \rangle$.
\end{lem}

\begin{proof} Suppose $3^{\max(0,2k-1-d)} \nmid [F:\Q]$, so $2k-1-d>0$.
First, suppose for the sake of contradiction that $r>k$. Then 
\[
\im \rho_{E,3^{r}} \neq \pi^{-1}(\im \rho_{E,3^k}),
\]
and by Lemma \ref{Lem3.8} we have $3^{2k-1-d} \mid [F:\Q]$. We have reached a contradiction. So $r \leq k$.

By assumption there exists $R \in E$ of order $3^r$ such that $\ker(\varphi) = \langle R \rangle$. With respect to the basis $\{3^{k-r}P, Q\}$ of $E[3^r]$, by Lemma \ref{Lem3.8} we may assume
\[
\im \rho_{E/F, 3^r} =
\left\{\begin{pmatrix}
1 & x \\
0 & y
\end{pmatrix}|\, x \in \Z/3^r\Z, y \in (\Z/3^r\Z)^{\times} \right\}.
\]
With respect to this basis, $R=\alpha 3^{k-r}P+ \beta Q$ for some $\alpha, \beta \in \Z/3^r\Z$. We will show $\beta=0$, from which we may conclude that $R \in \langle P \rangle$.

Since $R$ has order $3^r$, we must have $3 \nmid \alpha$ or $3 \nmid \beta$. First suppose $3 \nmid \alpha$. The $F$-rationality of $\langle R \rangle$ and description of $\im \rho_{E/F, 3^r}$ as above implies there exists $\sigma \in \Gal_F$ and $\gamma_1 \in (\Z/3^r\Z)^{\times}$ such that 
\[
\sigma(R)=
\begin{pmatrix}
1 & 0 \\
0 & 2
\end{pmatrix}\begin{pmatrix}
\alpha \\
\beta
\end{pmatrix}=
\begin{pmatrix}
\alpha \\
2\beta
\end{pmatrix}=
\begin{pmatrix}
\gamma_1 \alpha \\
\gamma_1 \beta
\end{pmatrix}.
\]
So $\gamma_1 =1$, which implies $\beta =0$. Now suppose $3 \nmid \beta$. As before, there must exist $\gamma_2 \in (\Z/3^r\Z)^{\times}$ such that 
\[
\begin{pmatrix}
1 & 1 \\
0 & 1
\end{pmatrix}\begin{pmatrix}
\alpha \\
\beta
\end{pmatrix}=
\begin{pmatrix}
\alpha+
\beta \\
\beta
\end{pmatrix}=
\begin{pmatrix}
\gamma_2 \alpha \\
\gamma_2 \beta
\end{pmatrix}.
\]
Again $\gamma_2=1$ and $\beta =0$.
\end{proof}

\begin{prop} \label{Prop5.7}
Suppose $F$ is a number field of odd degree and $E/F$ is an elliptic curve with $P\in E(F)$ of order $3^k$, $k \geq 2$. Let $\varphi:E \rightarrow E'$ be an $F$-rational isogeny, where there exists $E_0/\Q$ of with $j(E_0)=j(E')$ and $\im \rho_{E_0, 3^{\infty}}=9.36.0.6$ or $9.36.0.8$. Then $3^{2k-3} \mid [F:\Q]$. 
\end{prop}

\begin{proof}
By \cite[Lemma A.1]{CremonaNajmanQCurve}, we may assume $\varphi$ is cyclic and generated by a point of order $3^r\cdot n$ where $3 \nmid n$. If $n>1$, then by replacing $E$ with an $n$-isogenous curve if necessary, we may assume $\varphi$ has degree $3^r$.
If $\im \rho_{E_0, 3^{\infty}}=9.36.0.6$ or $9.36.0.8$, then $\ord_{3}([\GL_2(\Z_{3}): \im \rho_{E_0, 3^{\infty}}])=2$. Suppose for the sake of contradiction that $3^{2k-3} \nmid [F:\Q]$. By Lemma \ref{Lem3.10}, we have $r \leq k$ and $\ker(\varphi) \subseteq \langle P \rangle$.

First, suppose $r=k-1$ or $k$. Set $t=r+2$, and let $\{R,S\}$ be a basis of $E[3^{t}]$ such that $3^{t-k}R = P$ and $3^{t-k}S=Q$. 
Then we will show $\{\varphi(R), 3^{k-2}\varphi(Q)\}$ is a basis of $E'[9]$. Suppose there exist $\alpha, \beta\in \Z/9\Z$ such that $\alpha \varphi(R)=\beta 3^{k-2}\varphi(Q)$. This implies $\alpha R - \beta 3^{k-2} Q \in \ker(\varphi) \subseteq \langle P \rangle = \langle 3^{t-k}R \rangle$. Thus $\beta 3^{t-s}S$ is in the cyclic subgroup generated by $R$, and so $3^2 \mid \beta$ since $\{R,S\}$ is a basis of $E[3^{t}]$. Thus $\alpha \varphi(R)=\beta 3^{k-2}\varphi(Q)=\mathcal{O}$, as desired.

By Lemma \ref{Lem3.8} with respect to the basis $\{P,Q\}$, there exists $\sigma \in \Gal_F$ such that 
\[
\rho_{E/F,3^k}(\sigma)=\begin{pmatrix}
1 & 0 \\
0 & 4
\end{pmatrix}.
\] 
Also by Lemma \ref{Lem3.8} we have $\im \rho_{E/F,3^{\infty}} = \pi^{-1}(\im \rho_{E/F,3^k})$, so with respect to the basis $\{R,S\}$ of $E[3^{t}]$, we know there exists $\sigma' \in \Gal_F$ such that
\[
\rho_{E/F,3^{t}}(\sigma')=\begin{pmatrix}
1 & 0 \\
0 & 4
\end{pmatrix}.
\] 
Under the basis $\{\varphi(R), 3^{k-2}\varphi(Q)\}$ of $E'[9]$,
\begin{align*}
\sigma'(\varphi(R))&=\varphi(\sigma'(R))=\varphi(R)
\end{align*}
\begin{align*}
\sigma'(3^{k-2}\varphi(Q))&=3^{k-2}\varphi(\sigma'(Q))=4\cdot 3^{k-2}\varphi(Q).
\end{align*}
So
\[
\rho_{E'/F, 9}(\sigma')=
\begin{pmatrix}
1 & 0 \\
0 & 4
\end{pmatrix}.
\]

After at worst a quadratic extension $L/F$, we have $E'/L \cong_L E_0/L$. Since the matrix above has order 3, 
\[
\rho_{E_0/L, 9}(\sigma')=\begin{pmatrix}
1 & 0 \\
0 & 4
\end{pmatrix}.
\]
This means that the group generated by $\rho_{E_0/L, 9}(\sigma')$ is conjugate to an order 3 subgroup of 9.36.0.6 or 9.36.0.8 mod 9, and a \href{https://github.com/abbey-bourdon/minimal_torsion_curves/blob/main/Proposition6.6/code}{Magma computation} shows no such subgroup exists.

Now, suppose $r \leq k-2$. Then $\{3^{k-r-2} \varphi(P), 3^{k-2}\varphi(Q)\}$ is a basis of $E'[9]$ since $\ker(\varphi) \subseteq \langle P \rangle$. Since $P \in E(F)$, we have $3^{k-r-2} \varphi(P) \in E'(F)$. Moreover, for $\sigma \in \Gal_F$ as above,
\begin{align*}
\sigma(3^{k-2}\varphi(Q))&=3^{k-2}\varphi(\sigma(Q))=4 \cdot 3^{k-2}\varphi(Q),
\end{align*}
so
\[
\begin{pmatrix}
1 & 0 \\
0 & 4
\end{pmatrix}\in \im \rho_{E'/F,9}.
\]
We reach a contradiction as before.\end{proof}

\subsection{Proof of Proposition \ref{DivConditionsPrime3}}
Let $F=\Q(x)$. By Lemma \ref{Cor:IsogenyDef}, there is a model of $E/F$ where $P \in E(F)$ and such that there exists an $F$-rational cyclic isogeny $\varphi:E \rightarrow E'$ with $j(E')\in \Q$. Replacing $E$ with an isogenous curve if necessary, we may assume $\varphi$ has degree $3^r$. Since $E$ has a 3-isogeny over $F$, so does $E'$ by \cite[Proposition 3.2]{CremonaNajmanQCurve}. It follows from \cite[Proposition 3.3]{CremonaNajmanQCurve} that any $E_0/\Q$ with $j(E_0)=j(E')$ has a $\Q$-rational cyclic 3-isogeny, and $E_0$ gives a degree 1 closed point on $X_1(3)$ by Proposition \ref{IsogenyToRationalPoints}. By \cite[Corollary 1.3.1]{RSZ21}, the 3-adic image is one of the following groups, and we will consider each separately. Since we are interested in closed points on modular curves, and the degree of these points is not altered by taking quadratic twists, we may restrict to cases where $-I$ is contained in the 3-adic image. 
\begin{enumerate}
\item $\im \rho_{E_0, 3^{\infty}}=3.4.0.1$: Since $d=0$ and $\deg(X_1(3^k) \rightarrow X_1(3))=3^{2k-2}$, the divisibility condition of Proposition \ref{Prop1.6Restated} is best possible for all $k \in \Z^+$, and one can take $j_{min}=j(E_0)$.
\item $\im \rho_{E_0, 3^{\infty}}=3.12.0.1$ or $9.12.0.1$: If $\im \rho_{E_0, 3^{\infty}}=3.12.0.1$, then $E_0$ is 3-isogenous to an elliptic curve $E'/\Q$ with a rational cyclic 9-isogeny. By Proposition \ref{LadicIsogenyLemma}, the 3-adic Galois representation of $E'$ is completely determined by $\rho_{E_0, 3^{\infty}}$, and so it suffices to check a specific example. By viewing isogeny class \href{https://www.lmfdb.org/EllipticCurve/Q/175/b/}{175.b} in the LMFDB, we see that $\im \rho_{E', 3^{\infty}}=9.12.0.1$. Replacing $E_0$ with $E'$ if necessary, we are free to assume $E_0$ has image 9.12.0.1. Thus $E_0$ corresponds to closed points on $X_1(3)$ and $X_1(9)$ of degree 1 and 3, respectively. Since $d=1$ and $\deg(X_1(3^k) \rightarrow X_1(3))=3^{2k-2}$ for $k \geq 2$, the divisibility condition of Proposition \ref{Prop1.6Restated} is best-possible for $k \in \Z^+$ and we can take $j_{min}=j(E_0)$.
\item $\im \rho_{E_0, 3^{\infty}}=9.12.0.2$: A \href{https://github.com/abbey-bourdon/minimal_torsion_curves/blob/main/Proposition6.1/MagmaProp6.1_Part1.txt}{Magma computation} shows that for $E_0/\Q$ with this image, there exists a cubic extension $L$ such that $E_0/L$ has an $L$-rational 9-isogeny and an independent 3-isogeny. Thus over $L$, the curve $E_0$ is 3-isogenous to $E_1/L$ with a rational cyclic 27-isogeny. Thus $E_1$ gives a closed point on $X_1(27)$ of degree at most 27 by Proposition \ref{IsogenyToRationalPoints}. Since $d=1$ and $\deg(X_1(3^k) \rightarrow X_1(27))=3^{2k-6}$, the divisibility conditions of Proposition \ref{Prop1.6Restated} are best-possible for all $k \geq 3$ with $j_{min}=j(E_1)$. No elliptic curve in $\mathcal{E}$ with $j \in \Q$ has a point of order 27 in this degree or lower. 
\item $\im \rho_{E_0, 3^{\infty}}=9.36.0.2$ or $27.36.0.1$:  As in case (2), the isogeny class 304.c in the LMFDB shows we are free to assume $E_0$ has image 27.36.0.1. A \href{https://github.com/abbey-bourdon/minimal_torsion_curves/blob/main/Proposition6.1/MagmaProp6.1_Part1.txt}{Magma computation} shows that for $E_0/\Q$ with this image, there exists a cubic extension $L$ such that $E_0/L$ has an $L$-rational 27-isogeny and an independent 3-isogeny. Thus over $L$, the curve $E_0$ is 3-isogenous to $E_1/L$ with a rational cyclic 81-isogeny. Thus $E_1$ gives a closed point on $X_1(81)$ of degree at most 81 by Proposition \ref{IsogenyToRationalPoints}. Since $d=2$ and $\deg(X_1(3^k) \rightarrow X_1(81))=3^{2k-8}$, the divisibility conditions of Proposition \ref{Prop1.6Restated} are best-possible for all $k \geq 4$ with $j_{min}=j(E_1)$. No elliptic curve in $\mathcal{E}$ with $j$-invariant in $\Q$ has a point of order 81 in this degree or lower. 
\item $\im \rho_{E_0, 3^{\infty}}=9.36.0.3$ or $9.36.0.6$: As in case (2), the isogeny class \href{https://www.lmfdb.org/EllipticCurve/Q/22491/u/}{22491.u} in the LMFDB shows we are free to assume $E_0$ has image 9.36.0.6. Then $3^{\max(0,2k-3)} \mid [F:\Q]$ by Proposition \ref{Prop5.7}. The conclusion follows with $j_{min}=j(E_0)$.
\item $\im \rho_{E_0, 3^{\infty}}=9.36.0.1$, $9.36.0.4$, or $9.36.0.5$: As in case (2), the isogeny class 432.b in the LMFBD shows we are free to assume $E_0$ has image 9.36.0.4. A twist of $E_0$ has a rational point of order 9 and $d=2$, so divisibility conditions of Proposition \ref{Prop1.6Restated} are best possible for all $k \in \Z^+$ with $j_{min}=j(E_0)$.
\item $\im \rho_{E_0, 3^{\infty}}=9.36.0.7$ or $9.36.0.9$: As in case (2), the isogeny class \href{https://www.lmfdb.org/EllipticCurve/Q/1734/k/}{1734.k} in the LMFDB shows we are free to assume $E_0$ has image 9.36.0.7. A \href{https://github.com/abbey-bourdon/minimal_torsion_curves/blob/main/Proposition6.1/MagmaProp6.1_Part1.txt}{Magma computation} shows that there exists a cubic extension $L$ such that a twist $E_0^t$ of $E_0/L$ has an $L$-rational point of order 9 (say $Q$) and an independent 3-isogeny (say, with kernel generated by $R$). Then $\psi:E_0^t \rightarrow E_1=E_0^t/\langle R \rangle$ is a degree 3 isogeny, where $E_1$ has an $L$-rational cyclic 27-isogeny and $\psi(Q) \in E_1(F)$ is a point of order 9. Moreover, $\psi(Q)$ is in the kernel of the rational 27-isogeny. Thus the image of the 27-isogeny character $\chi$ associated to $E_1/L$ lands in $\{1,10,19\}$ and $E_1$ attains a point of order 27 in $\overline{L}^{\ker(\chi)}$, an extension of $L$ of degree dividing 3. Hence $E_1$ corresponds to a point on $X_1(27)$ of degree dividing 9. Since $d=2$, the divisibility conditions of Proposition \ref{Prop1.6Restated} are best-possible for $k \geq 3$ with $j_{min}=j(E_1)$. No elliptic curve in $\mathcal{E}$ with $j$-invariant in $\Q$ has a point of order 27 in this degree or lower. 
\item $\im \rho_{E_0, 3^{\infty}}=9.36.0.8$: By Proposition \ref{Prop5.7}, we have $3^{\max(0,2k-3)} \mid [F:\Q]$. A \href{https://github.com/abbey-bourdon/minimal_torsion_curves/blob/main/Proposition6.1/MagmaProp6.1_Part1.txt}{Magma computation} shows that for $E_0/\Q$ with this image, there exists a cubic extension $L$ such that $E_0/L$ has an $L$-rational 9-isogeny and an independent 3-isogeny. As in case (3), there exists $E_1$ which is 3-isogenous to $E_0$ and gives a closed point on $X_1(27)$ of degree at most 27. Thus the divisibility conditions of Proposition \ref{Prop5.7} are best-possible for $k \geq 3$ with $j_{min}=j(E_1)$. No elliptic curve in $\mathcal{E}$ with $j$-invariant in $\Q$ has a point of order 27 in this degree or lower. 
\end{enumerate} 
It remains to consider $k=2$ if $\im \rho_{E_0,3^{\infty}}=9.12.0.2, 9.36.0.2, 9.36.0.7$, or $9.36.0.8$ and $k=3$ if $\im \rho_{E_0,3^{\infty}}=9.36.0.2$. By Corollary \ref{ImageOnlyCor}, results can be obtained by choosing a particular elliptic curve $E_0/\Q$ with this Galois image and computing the degrees of closed points on $X_1(3^k)$ for elliptic curves $3^r$-isogenous to $E_0$, where the $j$-invariants of the isogenous curves are roots of the modular polynomial $\Phi_{3^r}(X,j(E_0))$. For $r$ sufficiently large, any odd degree point on $X_1(3^k)$ associated to an elliptic curve $3^r$-isogenous to $E_0$ will have degree divisible by $3 \delta$, so there are only finitely many curves to test. We do \href{https://github.com/abbey-bourdon/minimal_torsion_curves/blob/main/Proposition6.1/MagmaProp6.1_Part2.txt}{this computation} in Magma.

\section{Results for non-CM classes and $\ell  =2$}

Any non-CM $\Q$-curve defined over a number field of odd degree is geometrically isogenous to an elliptic curve with rational $j$-invariant, by work of Cremona and Najman \cite[Theorem 2.7]{CremonaNajmanQCurve}. Thus the following can be viewed as a strengthening of \cite[Proposition 4.1]{OddDegQCurve}, which shows that if a non-CM $\Q$-curve has a point of order $2^k$ over a field of odd degree then $k \leq 4$. There exist non-CM elliptic curves over $\Q$ with a rational point of order $8$, so the following gives the best possible bound. The following proposition proves Theorem \ref{Thm1.3} in the case of $\ell=2$.

\begin{proposition} \label{2case}
Let $\mathcal{E}$ be a rational $\overline{\Q}$-isogeny class of non-CM elliptic curves. If $E \in \mathcal{E}$ and $x=[E,P]\in X_1(2^k)$ is a point of odd degree, then $k \leq 3$. Moreover:
\begin{enumerate}
\item The least odd degree of a point on $X_1(2)$ associated to $\mathcal{E}$ is 1 or 3, depending on whether there exists $E'/\Q \in \mathcal{E}$ with a rational point of order 2.
\item If there exists an odd degree point on $X_1(4)$ associated to $\mathcal{E}$, then the least such degree is 1 or 3. The least degree is 1 if and only if there exists $E'/\Q \in \mathcal{E}$ with a rational point of order 4. The least odd degree is 3 if and only if there exists $E'/\Q \in \mathcal{E}$ full 2-torsion over a cubic field or with $\im \rho_{E', 2^{\infty}} =4.8.0.2$.
\item If there exists an odd degree point on $X_1(8)$ associated to $\mathcal{E}$, then the least such degree is 1. This occurs if and only if there exists $E'/\Q \in \mathcal{E}$ with a rational point of order 8.
\end{enumerate}
We can take $j_{min} \in \Q$ unless there exists $E'/\Q \in \mathcal{E}$ full 2-torsion over a cubic field, in which case $j_{min}$ is 2-isogenous to a rational $j$-invariant and defines a cubic extension.
\end{proposition}

\begin{proof}
Suppose for the sake of contradiction that $x\in X_1(16)$ is a point of odd degree, and let $F \coloneqq \Q(x)$. By Lemma \ref{Cor:IsogenyDef}, there exists a model of $E/F$ for which $P \in E(F)$ and such that there exists a rational cyclic isogeny $\varphi: E \rightarrow E'$ with $j(E') \in \Q$. If $\deg(\varphi)=2^r\cdot n$ for $n>1$ odd, we replace $E$ with an $n$-isogenous elliptic curve so we may assume $\varphi$ has degree $2^r$.

The dual isogeny $\hat{\varphi}: E' \rightarrow E$ is also cyclic of degree $2^r$, and so $E'$ possesses an $F$-rational cyclic subgroup of order $2^r$. Rational subgroups are twist-invariant, so any elliptic curve $E_0/\Q$ with $j(E_0)=j(E')$ will possess an $F$-rational subgroup of order $2^r$, say generated by $Q$. It follows that $2^{r-1} Q$ is $F$-rational. Since $F$ has odd degree, it must be that $[\Q(2^{r-1} Q) : \Q] = 1$ or 3. If $r=1$, then it follows $\Q(\langle Q \rangle)=\Q(2^{r-1} Q)$. Otherwise, by \cite[Proposition 3.6]{CremonaNajmanQCurve}, 
\[
[\Q(\langle Q \rangle):\Q(\langle 2Q \rangle)] \leq 2.
\]
Since $\Q(\langle Q \rangle)$ is contained if $F$, it must be of odd degree, and so $\Q(\langle Q \rangle)=\Q(\langle 2^{r-1}Q \rangle)=\Q(2^{r-1} Q)$. In either case, $\Q(\langle Q \rangle)$ is of degree 1 or 3. If $\Q(\langle Q \rangle)=\Q$, then $j(E) \in \Q$ and this contradicts \cite[Theorem 3]{OddDeg}. Thus $\Q(\langle Q \rangle)$ is of degree 3.

Next, we will show $r=1$. If not, then both $\langle 2^{r-2}Q\rangle$ and $2^{r-1} Q$ generate the same degree 3 extension of $\Q$. By Proposition \ref{IsogenyToRationalPoints}, the elliptic curve $E_0$ has a closed point of degree 3 on $X_1(4)$ lying above a degree 3 point on $X_1(2)$. By the classification of 2-adic images due to Rouse and Zureick-Brown \cite{RouseDZB}, this implies $E_0$ has 2-adic image X20b, X20a, or X20; see the data file associated to Corollaries 3.4 and 3.5 of \cite{GJLR}. These groups have RSZB labels 4.16.0.2, 8.16.0.3, and 4.8.0.2, respectively. Thus $\ord_2([\GL_2(\Z_2):\im \rho_{E_0,2^{\infty}}]) = 3$ or 4. However, this implies the degree of $F$ is even by Proposition \ref{Prop1.6Restated}. Therefore $r=1$ and $Q$ has order $2$. 

Thus $\varphi$ is of degree $2$, and $\varphi (P)$ is a point of order at least $8$ defined over $F$. This guarantees the existence of a closed point $y\in X_1(4)$ associated to $E'$ with $\deg(y)$ odd. Since $E_0$ gives a degree 3 point on $X_1(2)$, so does $E'$, and so $\deg(y)$ is an odd multiple of 3. As $\deg(X_1(4) \rightarrow X_1(2))=2$, this implies $\deg(y)=3$. But then we are again in the case where $E_0$ gives a closed point of degree 3 on $X_1(4)$ lying above a degree 3 point on $X_1(2)$. As above, $\ord_2([\GL_2(\Z_2):\im \rho_{E_0,2^{\infty}}]) = 3$ or 4 and we reach a contradiction via Proposition \ref{Prop1.6Restated}.

Finally, we will prove the refined degree bounds for each $k \leq 3$. As above, suppose $x \in X_1(2^k)$ is a point of odd degree associated to $E \in \mathcal{E}$ and $F \coloneqq \Q(x)$. We may assume there is an $F$-rational cyclic isogeny $\varphi:E \rightarrow E'$ of degree $2^r$ with $j(E')\in \Q$, where $E(F)$ has a point of order $2^k$. Let $E_0/\Q$ with $j(E_0)=j(E')$. By the second paragraph, if $E_0$ has a rational point of order 2, then $j(E) \in \Q$ and $E$ gives a closed point of degree 1 on $X_1(2)$ by \cite[Proposition 3.2]{CremonaNajmanQCurve}. By \cite[Proposition 4.6]{GJNajman}, the only odd degree closed points on $X_1(2^k)$ associated to $E$ have degree 1. So it suffices to consider the case when $E_0$ has no rational point of order 2.

The results for $X_1(2)$ are immediate, so suppose $2 \leq k \leq 3$. By \cite[Lemma 6.3]{OddDegQCurve}, the curve $E_0/\Q\in \mathcal{E}$ has full 2-torsion or a 4-isogeny defined over a number field of odd degree. Since $E_0(\Q)$ has no point of order 2, these conditions occur if $E_0$ has $\im \rho_{E_0,2} = \text{2Cn}$ or if $\im \rho_{E_0,2^{\infty}}\in \{4.16.0.2, 8.16.0.3, 4.8.0.2\}$, respectively, as above. In either case, $E_0$ or an elliptic curve 2-isogenous to $E_0$ has a cyclic 4-isogeny defined over a cubic field. This gives a degree 3 closed point on $X_1(4)$ by Proposition \ref{IsogenyToRationalPoints}. However, we will show no point on $X_1(8)$ associated to $\mathcal{E}$ has odd degree.

By replacing $E_0$ by a quadratic twist if necessary, we may assume the 2-adic image contains $-I$. Thus we may assume $\im \rho_{E_0,2^{\infty}}=2.2.0.1$ or 4.8.0.2; see \cite{RouseDZB} for curves that minimally cover X2 and X20. In the first, there are no odd degree points on $X_1(8)$ associated to $\mathcal{E}$ by Proposition \ref{Prop1.6Restated}, so assume $\im \rho_{E_0,2^{\infty}}=4.8.0.2$. By Corollary \ref{ImageOnlyCor}, results can be obtained by choosing a particular elliptic curve with this Galois image and computing the degrees of points on $X_1(8)$ for elliptic curves $2^r$-isogenous to $E_0$, where the $j$-invariants of the isogenous curves are roots of the modular polynomial $\Phi_{2^r}(X,j(E_0))$. A \href{https://github.com/abbey-bourdon/minimal_torsion_curves/blob/main/Proposition7.1/MagmaProp7.1.txt}{Magma computation} shows that 2 divides the degree of a point on $X_1(8)$ associated to any elliptic curve $2^r$-isogenous to $E_0$ for $r \in \Z^+$, as desired.
\end{proof}

\section{$\ell$-adic images of level $\ell$}

\begin{prop} \label{LevelEllImages}
Let $\mathcal{E}$ be a rational $\overline{\Q}$-isogeny class of non-CM elliptic curves and let $\ell$ be prime. Suppose there exists $E_0/\Q \in \mathcal{E}$ with $\ell$-adic Galois representation of level $\ell$. Among points on $X_1(\ell^k)$ associated to $\mathcal{E}$, a point of least degree can always be associated to $j_{min} \in \mathcal{E}$ which is at most $\ell$-isogenous to a rational $j$-invariant.
\end{prop}

\begin{proof}
If $\im \rho_{E_0,\ell}$ is surjective, this follows from Proposition \ref{Prop1.6Restated} and the formula for $\deg(X_1(\ell^k) \rightarrow X_1(\ell))$; see Proposition \ref{prop:Degree}. Suppose $\ell$ is odd. If $E_0/\Q$ has a rational $\ell$-isogeny, then the result follow from Lemma \ref{MinTorsionCurveForIsogeny} if $\ell \geq 5$ and the proof of Proposition \ref{DivConditionsPrime3} if $\ell=3$; see cases 1 and 2. .

If $E_0/\Q$ has no rational $\ell$-isogeny and $\im \rho_{E_0,\ell}$ is not surjective, then $\im \rho_{E_0, \ell^{\infty}}$ is the complete preimage of one of the following groups (see $\S2.2$). One may check that $\ord_{\ell}([\GL_2(\Z_{\ell}): \im \rho_{E_0, \ell^{\infty}}])=1$. We will consider each case separately. We will see $j_{min} \not\in \Q$ in each case.
\begin{itemize}
\item $C_{ns}^+(\ell)$: This is a subgroup of order $2(\ell^2-1)$, and up to a choice of basis it contains all matrices
\[
\begin{pmatrix}
a & 0  \\
0  & a
\end{pmatrix},  a \not\equiv 0 \pmod{\ell},
\]
\[
\begin{pmatrix}
a & 0  \\
0  & -a
\end{pmatrix},  a \not\equiv 0 \pmod{\ell}.
\]
Since $\ell$ is odd, these matrices form a group of order $2(\ell-1)$. Its fixed field has size $\ell+1$, so over an extension of degree $\ell+1$, the curve $E_0$ attains two independent $\ell$-isogenies with kernels $C_1$ and $C_2$. Then by Proposition \ref{IsogenyToRationalPoints}, the curve $E_0/C_1$ attains a closed point on $X_1(\ell)$ in degree dividing $(\ell+1)\cdot \varphi(\ell)/2=(\ell^2-1)/2$ and a closed point on $X_1(\ell^2)$ in degree dividing $(\ell+1)\cdot \varphi(\ell^2)/2=(\ell^2-1)\cdot \ell/2$. Since all closed points on $X_1(\ell)$ associated to $E_0$ have degree $(\ell^2-1)/2$, the claim holds by Propositions \ref{Prop1.6Restated} and \ref{prop:Degree} with $j_{min}=j(E_0/C_1)$.

\item 13S4, 5S4: A \href{https://github.com/abbey-bourdon/minimal_torsion_curves/blob/main/Proposition8.1/code}{Magma computation} shows the curve $E_0$ attains two independent $\ell$-isogenies in degree 6 with kernels $C_1$ and $C_2$. By Proposition \ref{IsogenyToRationalPoints}, the curve $E_0/C_1$ attains a closed point on $X_1(\ell)$ in degree dividing $6 \cdot \frac{\varphi(\ell)}{2}=3(\ell-1)$ and a closed point on $X_1(\ell^2)$ in degree dividing $6\cdot \frac{\varphi(\ell^2)}{2}=3 \ell(\ell-1)$. The claim holds by Propositions \ref{Prop1.6Restated} and \ref{prop:Degree} with $j_{min}=j(E_0/C_1)$.
\item 7Ns, 7Ns.2.1, 7Ns.3.1, 5Ns, 5Ns.2.1, 3Ns: The curve $E_0$ picks up 2 independent $\ell$-isogenies in degree 2 with kernels $C_1$ and $C_2$. By Proposition \ref{IsogenyToRationalPoints}, the curve $E_0/C_1$ attains a closed point on $X_1(\ell)$ in degree dividing $2 \cdot \frac{\varphi(\ell)}{2}=\ell-1$ and a closed point on $X_1(\ell^2)$ in degree dividing $2\cdot \frac{\varphi(\ell^2)}{2}= \ell(\ell-1)$. The claim holds by Propositions \ref{Prop1.6Restated} and \ref{prop:Degree} with $j_{min}=j(E_0/C_1)$.
\end{itemize}

Now suppose $\ell=2$. If $\im \rho_{E_0,2}=$ 2Cs, then $E_0$ has full 2-torsion over $\Q$. Hence it is isogenous over $\Q$ to an elliptic curve $E'/\Q$ with a $\Q$-rational cyclic 4-isogeny. By Proposition \ref{IsogenyToRationalPoints}, there is a degree 1 closed point on $X_1(4)$ associated to $E'$, and the claim follows from Proposition \ref{Prop1.6Restated} and the formula for $\deg(X_1(2^k) \rightarrow X_1(4))$. If $\im \rho_{E_0,2}=$ 2B or 2Cn, then $E_0$ has full 2-torsion over an extension $K$ of degree 2 or 3, respectively. There is $E'/K$ with a $K$-rational cyclic 4-isogeny, and the claim follows as in the previous case. 
\end{proof}

\begin{remark}
The expression for the least degree does not necessarily divide all degrees. For example, the proof of Proposition \ref{LevelEllImages} shows that the least degree of a point on $X_1(7^k)$ associated to $\mathcal{E}$ containing $E_0/\Q$ with $\im \rho_{E_0,7^{\infty}}=$ 7.28.0.1 is $7^{\max(0,2k-3)}\cdot 6$. However, there is a closed point on $X_1(7)$ associated to $E_0$ of degree 9, and points on $X_1(7^k)$ lying above this point will not have degree divisible by $7^{\max(0,2k-3)}\cdot 6$.
\end{remark}

The next result shows that working with curves in a non-CM rational geometric isogeny class $\mathcal{E}$ can yield points of strictly lower degree than those associated with any rational non-CM $j$-invariant. This relies on recent work of Furio \cite{Furio2024}, as applied in work of the first author with Ejder \cite{BourdonEjder}.

\begin{cor}\label{LowerDegCor}
The least degree of a point on $X_1(49)$ associated to any non-CM rational $\overline{\Q}$-isogeny class $\mathcal{E}$ is at most 42, whereas the least degree of a non-CM point $x \in X_1(49)$ with $j(x)\in\Q$ is 49.
\end{cor}

\begin{proof}
The first claim follows from the previous proposition: Consider, for example, an elliptic curve over $\Q$ with mod 7 image 7Ns. For the second claim, note that all known non-surjective 7-adic images for non-CM elliptic curves $E/\Q$ have level 7; see \cite{RouseDZB}. Thus any point on $X_1(49)$ associated to such an $E$ has degree at least 49. Equality holds for any elliptic curve over $\Q$ with a rational point of order 7. If $E/\Q$ has surjective 7-adic image, then $E$ gives a single closed point on $X_1(49)$ of degree 1176. By \cite[Theorem 1.1.6]{RSZ21} and \cite[Theorem 1.4]{FurioLombardo25}, other groups which could occur as 7-adic images for non-CM elliptic curves over $\Q$ would give rise to rational points on the modular curves with RSZB label 49.147.9.1 or 49.196.9.1. In the first case, such an $E/\Q$ would have mod 7 image landing in the normalizer of a non-split Cartan subgroup. Possible 7-adic images for $E$ are given in \cite[Theorem 1.9]{Furio2024}, and the last paragraph of the proof of Theorem 6 in \cite{BourdonEjder} shows $E$ gives a point on $X_1(49)$ of degree at least 168. Finally, if $E/\Q$ corresponds to a rational point on the modular curve labeled 49.196.9.1, then $\im \rho_{E,7^{\infty}}=49.196.9.1$ by \cite[Corollary 3]{BourdonEjder}. Such an image gives points on $X_1(49)$ of degree at least 294.
\end{proof}

\section{CM Elliptic Curves}
Let $\mathcal{E}$ be a $\overline{\Q}$-isogeny class of CM elliptic curves. The endomorphism algebra $K=\text{End}(E) \otimes \Q$ is an isogeny invariant, so all elliptic curves in $\mathcal{E}$ have CM by an order in the imaginary quadratic field $K$. In fact, since any CM elliptic curve is isogenous to one with CM by the maximal order (see, for example, \cite[Proposition 2.2]{BP}), the class $\mathcal{E}$ contains elliptic curves with CM by any possible order in $K$. In this section, we study the isogeny distance from a minimal torsion curve to an elliptic curve $E$ with CM by the full ring of integers in $K$, since $[\Q(j(E)):\Q]$ is minimal for $\mathcal{E}$. This builds on work of the first author and Clark \cite{BC1,BC2}. A key first step is to establish sharp lower bounds on the least degree of a point on $X_1(\ell^k)$ associated to $\mathcal{E}$. These appear as Propositions \ref{CaseSplit}, \ref{CaseInert}, and \ref{CaseRamified}. Taken together they imply Theorem \ref{CMthm}. Preliminary results about CM elliptic curves are summarized in Section \ref{Section2.5}. 

Throughout this section $w_K = \#\OO_K^\times$ and $h_K$ denotes the class number of $K$. 

\subsection{$\ell$ split in $K$}
\begin{prop}\label{CaseSplit}
Let $\mathcal{E}$ be a $\overline{\Q}$-isogeny class of elliptic curves with CM by orders in the imaginary quadratic field $K$, and let $\ell$ be a prime split in $K$. Then the least degree of a point on $X_1(\ell^n)$ associated to $\mathcal{E}$ is 
\[
2\cdot h_K \cdot \ell^{n-1}(\ell-1)/w_K,
\]
and this is attained by $E\in \mathcal{E}$ with CM by the maximal order in $K$. Thus $[\Q(j_{min}):\Q]=h_K$.
\end{prop}

\begin{proof}
That such an $E$ gives a point on $X_1(\ell^n)$ in this degree follows from \cite[Theorem 6.2]{BC2} and $[\Q(j(E)):\Q]=h_K$ by Section \ref{Section2.5}; note that $\ell > 3$ if $\Delta=-3,-4$ by the assumption that $\ell$ is split in $K$. It remains to show this is the least possible degree among all $E' \in \mathcal{E}$. We may assume $\ell^n>2$. Since the endomorphism algebra is an isogeny invariant, any $E' \in \mathcal{E}$ has CM by an order in $K$. Since we already have the least degree for a point with CM by the maximal order, we will henceforth assume $E'$ has CM by an order in $K$ of conductor $\mathfrak{f}>1$. For any point $x=[E',P'] \in X_1(\ell^n)$, by \cite[Theorem 6.2]{BC1} we have
\[
h_K \cdot \frac{\ell^{n-1}(\ell-1)}{2} \mid  \deg(x).
\]
If $\deg(x)=h_K \cdot \frac{\ell^{n-1}(\ell-1)}{2}\cdot d<2\cdot h_K \cdot \frac{\ell^{n-1}(\ell-1)}{w_K}$ for some $d \in \Z^+$, it must be that $d=1$ and $w_K=2$.
This implies $[\Q(j(E')):\Q]=h_K$. The degree of this extension is equal to the class number of the order $\OO$; see Equation \ref{ClassNo} in Section \ref{Section2.5}. Since $w_K=2$, then $[\Q(j(E')):\Q]=h_K$ implies $E'$ has CM by an order in $K$ of conductor dividing 2. Since we have assumed $E'$ has CM by an order of conductor $\mathfrak{f}>1$, we will suppose $E'$ has CM by the order in $K$ of conductor 2. By Equation \ref{ClassNo}, this can happen only if 2 is split in $K$. But this contradicts \cite[Theorem 6.2]{BC2} if $\ell$ is odd and \cite[Theorem 6.6]{BC2} if $\ell=2$.
\end{proof}

\subsection{$\ell$ inert in $K$}

\begin{prop} \label{CaseInert}
Let $\mathcal{E}$ be a $\overline{\Q}$-isogeny class of elliptic curves with CM by orders in the imaginary quadratic field $K$, and let $\ell$ be a prime inert in $K$. The least degree of a point on $X_1(\ell^n)$ associated to $\mathcal{E}$ is 
\[
\delta \coloneqq \begin{cases} h_K \cdot \ell^{\lfloor{3(n-1)/2}\rfloor+1}(\ell^2-1)/w_K \text{ if $\ell=2$,} \\ h_K \cdot \ell^{\lfloor{3(n-1)/2}\rfloor}(\ell^2-1)/w_K \text{ if $\ell \geq 3$.} \end{cases}
\]
This is attained by $E\in \mathcal{E}$ with CM by an order in $K$ of conductor  $\mathfrak{f}=\ell^{\lfloor{n/2}\rfloor}$. Moreover, if $j_{min}$ is the $j$-invariant of a minimal torsion curve of level $\ell^n$ and $n \geq 5$, then
\[
[\Q(j_{min}):\Q] \geq h_K \cdot \frac{\ell^{(n-5)/2}}{3}(\ell+1).
\]
Thus $[\Q(j_{min}):\Q] \rightarrow \infty$ as $n \rightarrow \infty$.
\end{prop}

\begin{proof}
Suppose $E$ has CM by the order in $K$ of conductor  $\mathfrak{f}=\ell^{\lfloor{n/2}\rfloor}$. Note we can find such an $E\in \mathcal{E}$ by the first paragraph of $\S9$. Then by \cite[Theorem 6.1, Theorem 6.6]{BC2}, the point $x \in X_1(\ell^n)$ of least degree associated to such an $E$ has
\[
\deg(x)=2^{\epsilon} \cdot T(\OO, \ell^n)\cdot h(\OO),
\]
where $T(\OO, \ell^n)$ is as defined in \cite[Theorem 4.1]{BC2} and $\epsilon=1$ if $\ell=2$, $n>1$ and $\epsilon=0$ otherwise. Evaluating $T(\OO,\ell^n)$ via \cite[Theorem 4.1]{BC2} and $h(\OO)$ with Equation \ref{ClassNo} shows $\deg(x)=\delta$.

Now we will justify that this is the least possible degree of a point on $X_1(\ell^n)$ associated to $\mathcal{E}$. Suppose $E' \in \mathcal{E}$ has CM by the order of conductor $\ell^c\mathfrak{f}'$ in $\OO_K$ where $\ell \nmid \mathfrak{f}'$. By \cite[Theorem 4.1, Theorem 6.1]{BC2} and Equation \ref{ClassNo}, the least degree of a point on $X_1(\ell^n)$ associated to $E'$ is at least $\delta$, and this inequality is strict if $c <  \frac{n-3}{2}$. Thus, any minimal torsion curve of level $\ell^n$ must have $c \geq \frac{n-3}{2}$. If $n \geq 5$, then $[\Q(j_{min}):\Q] \geq h_K \cdot \frac{\ell^{(n-5)/2}}{3}(\ell+1)$ by Equation \ref{ClassNo} in Section 2.5. The conclusion follows. 
\end{proof}

\subsection{$\ell$ ramified in $K$}

\begin{prop} \label{CaseRamified}
Let $\mathcal{E}$ be a $\overline{\Q}$-isogeny class of elliptic curves with CM by orders in the imaginary quadratic field $K$, and let $\ell$ be a prime ramified in $K$. Then the least degree of a point on $X_1(\ell^n)$ associated to $\mathcal{E}$ is 
\[
\delta \coloneqq \begin{cases} h_K \text{ if $\ell^n\leq3$}, \\ h_K \cdot \ell^{\lfloor{3(n-1)/2}\rfloor+1}(\ell-1)/w_K \text{ if $\ell=2,n>1, \ord_2(\Delta_K)=2$}, \\ h_K \cdot \ell^{\lfloor{3n/2}\rfloor-1}(\ell-1)/w_K \text{ otherwise}. \end{cases}
\]
The least degree is attained by $E\in \mathcal{E}$ with CM by an order in $K$ of conductor  $\mathfrak{f}=\ell^{\lfloor{n/2}\rfloor}$. Also, for any $j$-invariant $j_{min}$ of a minimal torsion curve of level $\ell^n$ and $n \geq 5$, one has 
\[
[\Q(j_{min}):\Q] \geq h_K\cdot \frac{\ell^{(n-3)/2}}{3}.
\]
Thus $[\Q(j_{min}):\Q] \rightarrow \infty$ as $n \rightarrow \infty$.
\end{prop}

\begin{proof}
Suppose $E$ has CM by an order in $K$ of conductor  $\mathfrak{f}=\ell^{\lfloor{n/2}\rfloor}$; we can find such an $E\in \mathcal{E}$ by the first paragraph of $\S9$. Then by \cite[Theorem 6.6]{BC2} the least degree $x \in X_1(\ell^n)$ associated to $E$ is
\[
\deg(x)=2^{\epsilon}\cdot T(\OO, \ell^n)\cdot h(\OO),
\]
where $T(\OO, \ell^n)$ is as defined in \cite[Theorem 4.1]{BC2} and $\epsilon =1$ if $\ord_2(\Delta_K)=2, \ell=2,n$ an odd integer greater than 1 and $\epsilon=0$ otherwise. Evaluating $T(\OO,\ell^n)$ via \cite[Theorem 4.1]{BC2} and replacing $h(\OO)$ with the formula in Equation \ref{ClassNo} of $\S2.5$ shows $\deg(x)=\delta$. 

Now we will justify that this is the least possible degree of a point on $X_1(\ell^n)$ associated to $\mathcal{E}$. Suppose $E' \in \mathcal{E}$ has CM by the order of conductor $\ell^c\mathfrak{f}'$ in $\OO_K$ where $\ell \nmid \mathfrak{f}'$. By \cite[Theorem 4.1, Theorem 6.6]{BC2} and Equation \ref{ClassNo}, the least degree of a point on $X_1(\ell^n)$ associated to $E'$ is at least $\delta$, and this inequality is strict if $c < \frac{n-3}{2}$. Thus any minimal torsion curve of level $\ell^n$ must have $c \geq \frac{n-3}{2}$, meaning $[\Q(j_{min}):\Q] \geq h_K \cdot \frac{\ell^{(n-3)/2}}{3}$ by Equation \ref{ClassNo} in Section 2.5. The conclusion follows.
\end{proof}

\bibliographystyle{amsplain}
\bibliography{bibliography1}

@book {cox,
    AUTHOR = {Cox, David A.},
     TITLE = {Primes of the form {$x^2 + ny^2$}},
    SERIES = {A Wiley-Interscience Publication},
      NOTE = {Fermat, class field theory and complex multiplication},
 PUBLISHER = {John Wiley \& Sons Inc.},
   ADDRESS = {New York},
      YEAR = {1989},
     PAGES = {xiv+351},
      ISBN = {0-471-50654-0; 0-471-19079-9},
   MRCLASS = {11A41 (11F11 11R11 11R16 11R18 11R37 11Y11)},
MRREVIEWER = {Andrew Bremner},
}

@book {shimura,
    AUTHOR = {Shimura, Goro},
     TITLE = {Introduction to the arithmetic theory of automorphic
              functions},
      NOTE = {Kan{\^o} Memorial Lectures, No. 1},
 PUBLISHER = {Publications of the Mathematical Society of Japan, No. 11.
              Iwanami Shoten, Publishers, Tokyo},
      YEAR = {1971},
     PAGES = {xiv+267},
   MRCLASS = {10D10 (12-02 14K22)},
MRREVIEWER = {A. N. Andrianov},
}

@article{heilbronn,
author = {Heilbronn, Hans}, 
title = {ON THE CLASS-NUMBER IN IMAGINARY QUADRATIC FIELDS},
volume = {os-5}, 
number = {1}, 
pages = {150-160}, 
year = {1934}, 
doi = {10.1093/qmath/os-5.1.150}, 
URL = {http://qjmath.oxfordjournals.org/content/os-5/1/150.short}, 
eprint = {http://qjmath.oxfordjournals.org/content/os-5/1/150.full.pdf+html}, 
journal = {The Quarterly Journal of Mathematics} 
}

@article {mazur77,
    AUTHOR = {Mazur, B.},
     TITLE = {Modular curves and the {E}isenstein ideal},
   JOURNAL = {Inst. Hautes \'Etudes Sci. Publ. Math.},
  FJOURNAL = {Institut des Hautes \'Etudes Scientifiques. Publications
              Math\'ematiques},
    NUMBER = {47},
      YEAR = {1977},
     PAGES = {33--186 (1978)},
      ISSN = {0073-8301},
     CODEN = {PMIHA6},
MRREVIEWER = {M. Ohta},
       URL = {http://www.numdam.org/item?id=PMIHES_1977__47__33_0},
}

@article {mazur78,
    AUTHOR = {Mazur, B.},
     TITLE = {Rational isogenies of prime degree (with an appendix by {D}.
              {G}oldfeld)},
   JOURNAL = {Invent. Math.},
  FJOURNAL = {Inventiones Mathematicae},
    VOLUME = {44},
      YEAR = {1978},
    NUMBER = {2},
     PAGES = {129--162},
      ISSN = {0020-9910},
     CODEN = {INVMBH},
   MRCLASS = {14K07 (10D35 14G25)},
MRREVIEWER = {V. V. Shokurov},
       DOI = {10.1007/BF01390348},
       URL = {http://dx.doi.org/10.1007/BF01390348},
}

@article {kamienny86,
    AUTHOR = {Kamienny, S.},
     TITLE = {Torsion points on elliptic curves over all quadratic fields},
   JOURNAL = {Duke Math. J.},
  FJOURNAL = {Duke Mathematical Journal},
    VOLUME = {53},
      YEAR = {1986},
    NUMBER = {1},
     PAGES = {157--162},
      ISSN = {0012-7094},
     CODEN = {DUMJAO},
MRREVIEWER = {Loren D. Olson},
       DOI = {10.1215/S0012-7094-86-05310-X},
       URL = {http://dx.doi.org/10.1215/S0012-7094-86-05310-X},
}

@article {kamienny92,
    AUTHOR = {Kamienny, S.},
     TITLE = {Torsion points on elliptic curves and {$q$}-coefficients of
              modular forms},
   JOURNAL = {Invent. Math.},
  FJOURNAL = {Inventiones Mathematicae},
    VOLUME = {109},
      YEAR = {1992},
    NUMBER = {2},
     PAGES = {221--229},
      ISSN = {0020-9910},
     CODEN = {INVMBH},
MRREVIEWER = {Glenn Stevens},
       DOI = {10.1007/BF01232025},
       URL = {http://dx.doi.org/10.1007/BF01232025},
}

@article {KM88,
    AUTHOR = {Kenku, M. A. and Momose, F.},
     TITLE = {Torsion points on elliptic curves defined over quadratic
              fields},
   JOURNAL = {Nagoya Math. J.},
  FJOURNAL = {Nagoya Mathematical Journal},
    VOLUME = {109},
      YEAR = {1988},
     PAGES = {125--149},
      ISSN = {0027-7630},
     CODEN = {NGMJA2},
MRREVIEWER = {Bert van Geemen},
       URL = {http://projecteuclid.org/euclid.nmj/1118780896},
}

@article {serre72,
    AUTHOR = {Serre, Jean-Pierre},
     TITLE = {Propri\'et\'es galoisiennes des points d'ordre fini des
              courbes elliptiques},
   JOURNAL = {Invent. Math.},
  FJOURNAL = {Inventiones Mathematicae},
    VOLUME = {15},
      YEAR = {1972},
    NUMBER = {4},
     PAGES = {259--331},
      ISSN = {0020-9910},
   MRCLASS = {14G25 (14K15)},
MRREVIEWER = {J. W. S. Cassels},
}

@preamble{
   "\def\cprime{$'$} "
}

@book {silverman,
    AUTHOR = {Silverman, Joseph H.},
     TITLE = {The arithmetic of elliptic curves},
    SERIES = {Graduate Texts in Mathematics},
    VOLUME = {106},
   EDITION = {Second},
 PUBLISHER = {Springer, Dordrecht},
      YEAR = {2009},
     PAGES = {xx+513},
      ISBN = {978-0-387-09493-9},
   MRCLASS = {11-02 (11G05 11G20 14H52 14K15)},
MRREVIEWER = {Vasil{\cprime} {\=I}. Andr{\={\i}}{\u\i}chuk},
       DOI = {10.1007/978-0-387-09494-6},
       URL = {http://dx.doi.org/10.1007/978-0-387-09494-6},
}

@incollection {rubin,
    AUTHOR = {Rubin, Karl},
     TITLE = {Elliptic curves with complex multiplication and the conjecture
              of {B}irch and {S}winnerton-{D}yer},
 BOOKTITLE = {Arithmetic theory of elliptic curves ({C}etraro, 1997)},
    SERIES = {Lecture Notes in Math.},
    VOLUME = {1716},
     PAGES = {167--234},
 PUBLISHER = {Springer, Berlin},
      YEAR = {1999},
   MRCLASS = {11G40 (11G05)},
MRREVIEWER = {Fernando Rodr{\'{\i}}guez Villegas},
       DOI = {10.1007/BFb0093455},
       URL = {http://dx.doi.org/10.1007/BFb0093455},
}

@book {modular,
    AUTHOR = {Diamond, Fred and Shurman, Jerry},
     TITLE = {A first course in modular forms},
    SERIES = {Graduate Texts in Mathematics},
    VOLUME = {228},
 PUBLISHER = {Springer-Verlag, New York},
      YEAR = {2005},
     PAGES = {xvi+436},
      ISBN = {0-387-23229-X},
   MRCLASS = {11Fxx},
MRREVIEWER = {Henri Darmon},
}

@article {BCS,
    AUTHOR = {Bourdon, Abbey and Clark, Pete L. and Stankewicz, James},
     TITLE = {Torsion points on {CM} elliptic curves over real number
              fields},
   JOURNAL = {Trans. Amer. Math. Soc.},
  FJOURNAL = {Transactions of the American Mathematical Society},
    VOLUME = {369},
      YEAR = {2017},
    NUMBER = {12},
     PAGES = {8457--8496},
      ISSN = {0002-9947},
   MRCLASS = {11G05 (11G15)},
MRREVIEWER = {Jie Shu},
       DOI = {10.1090/tran/6905},
       URL = {https://doi.org/10.1090/tran/6905},
}

@article {BELOV,
    AUTHOR = {Bourdon, Abbey and Ejder, \"{O}zlem and Liu, Yuan and Odumodu,
              Frances and Viray, Bianca},
     TITLE = {On the level of modular curves that give rise to isolated
              {$j$}-invariants},
   JOURNAL = {Adv. Math.},
  FJOURNAL = {Advances in Mathematics},
    VOLUME = {357},
      YEAR = {2019},
     PAGES = {106824, 33},
      ISSN = {0001-8708},
   MRCLASS = {14G35 (11G05)},
       DOI = {10.1016/j.aim.2019.106824},
       URL = {https://doi.org/10.1016/j.aim.2019.106824},
}

@article {GJNajman,
    AUTHOR = {Gonz\'{a}lez-Jim\'{e}nez, Enrique and Najman, Filip},
     TITLE = {Growth of torsion groups of elliptic curves upon base change},
   JOURNAL = {Math. Comp.},
  FJOURNAL = {Mathematics of Computation},
    VOLUME = {89},
      YEAR = {2020},
    NUMBER = {323},
     PAGES = {1457--1485},
      ISSN = {0025-5718},
   MRCLASS = {11G05},
  MRNUMBER = {4063324},
MRREVIEWER = {Daniel Sadornil},
       DOI = {10.1090/mcom/3478},
       URL = {https://doi.org/10.1090/mcom/3478},
}

@article {BC1,
    AUTHOR = {Bourdon, Abbey and Clark, Pete L.},
     TITLE = {Torsion points and {G}alois representations on {CM} elliptic
              curves},
   JOURNAL = {Pacific J. Math.},
  FJOURNAL = {Pacific Journal of Mathematics},
    VOLUME = {305},
      YEAR = {2020},
    NUMBER = {1},
     PAGES = {43--88},
      ISSN = {0030-8730},
   MRCLASS = {11G05 (11G15)},
       DOI = {10.2140/pjm.2020.305.43},
       URL = {https://doi.org/10.2140/pjm.2020.305.43},
}

@article {BC2,
    AUTHOR = {Bourdon, Abbey and Clark, Pete L.},
     TITLE = {Torsion points and isogenies on {CM} elliptic curves},
   JOURNAL = {J. Lond. Math. Soc. (2)},
  FJOURNAL = {Journal of the London Mathematical Society. Second Series},
    VOLUME = {102},
      YEAR = {2020},
    NUMBER = {2},
     PAGES = {580--622},
      ISSN = {0024-6107},
   MRCLASS = {11G05 (11G15)},
       DOI = {10.1112/jlms.12329},
       URL = {https://doi.org/10.1112/jlms.12329},
}

@article {BP,
    AUTHOR = {Bourdon, Abbey and Pollack, Paul},
     TITLE = {Torsion subgroups of {CM} elliptic curves over odd degree
              number fields},
   JOURNAL = {Int. Math. Res. Not. IMRN},
  FJOURNAL = {International Mathematics Research Notices. IMRN},
      YEAR = {2017},
    NUMBER = {16},
     PAGES = {4923--4961},
      ISSN = {1073-7928},
   MRCLASS = {11G05},
MRREVIEWER = {M\'{a}rton Erd\'{e}lyi},
       DOI = {10.1093/imrn/rnw163},
       URL = {https://doi.org/10.1093/imrn/rnw163},
}

@article {greenberg2012,
    AUTHOR = {Greenberg, Ralph},
     TITLE = {The image of {G}alois representations attached to elliptic
              curves with an isogeny},
   JOURNAL = {Amer. J. Math.},
  FJOURNAL = {American Journal of Mathematics},
    VOLUME = {134},
      YEAR = {2012},
    NUMBER = {5},
     PAGES = {1167--1196},
      ISSN = {0002-9327},
     CODEN = {AJMAAN},
   MRCLASS = {11F80 (11G05)},
MRREVIEWER = {Ravi K. Ramakrishna},
       DOI = {10.1353/ajm.2012.0040},
       URL = {http://dx.doi.org/10.1353/ajm.2012.0040},
}

@article {greenberg2014,
    AUTHOR = {Greenberg, R. and Rubin, K. and Silverberg, A. and Stoll, M.},
     TITLE = {On elliptic curves with an isogeny of degree 7},
   JOURNAL = {Amer. J. Math.},
  FJOURNAL = {American Journal of Mathematics},
    VOLUME = {136},
      YEAR = {2014},
    NUMBER = {1},
     PAGES = {77--109},
      ISSN = {0002-9327},
   MRCLASS = {11G05 (14H52)},
       DOI = {10.1353/ajm.2014.0005},
       URL = {http://dx.doi.org/10.1353/ajm.2014.0005},
}

@article {SutherlandZywina,
    AUTHOR = {Sutherland, Andrew V. and Zywina, David},
     TITLE = {Modular curves of prime-power level with infinitely many
              rational points},
   JOURNAL = {Algebra Number Theory},
  FJOURNAL = {Algebra \& Number Theory},
    VOLUME = {11},
      YEAR = {2017},
    NUMBER = {5},
     PAGES = {1199--1229},
      ISSN = {1937-0652},
   MRCLASS = {14G35 (11F80 11G05 14G05)},
MRREVIEWER = {Jie Shu},
       DOI = {10.2140/ant.2017.11.1199},
       URL = {https://doi.org/10.2140/ant.2017.11.1199},
}

@article {sutherland,
    AUTHOR = {Sutherland, Andrew V.},
     TITLE = {Computing images of {G}alois representations attached to
              elliptic curves},
   JOURNAL = {Forum Math. Sigma},
  FJOURNAL = {Forum of Mathematics. Sigma},
    VOLUME = {4},
      YEAR = {2016},
     PAGES = {e4, 79},
      ISSN = {2050-5094},
   MRCLASS = {11G05 (11F80 11G20 11Y16)},
MRREVIEWER = {Susan L. Schmoyer},
       DOI = {10.1017/fms.2015.33},
       URL = {https://doi.org/10.1017/fms.2015.33},
}

@article {GJLR,
    AUTHOR = {Gonz\'{a}lez-Jim\'{e}nez, Enrique and Lozano-Robledo, \'{A}lvaro},
     TITLE = {On the minimal degree of definition of {$p$}-primary torsion
              subgroups of elliptic curves},
   JOURNAL = {Math. Res. Lett.},
  FJOURNAL = {Mathematical Research Letters},
    VOLUME = {24},
      YEAR = {2017},
    NUMBER = {4},
     PAGES = {1067--1096},
      ISSN = {1073-2780},
   MRCLASS = {11G05 (14H52)},
MRREVIEWER = {Wei Pin Wong},
       DOI = {10.4310/MRL.2017.v24.n4.a7},
       URL = {https://doi.org/10.4310/MRL.2017.v24.n4.a7},
}

@article {RouseDZB,
    AUTHOR = {Rouse, Jeremy and Zureick-Brown, David},
     TITLE = {Elliptic curves over {$\Bbb Q$} and 2-adic images of {G}alois},
   JOURNAL = {Res. Number Theory},
  FJOURNAL = {Research in Number Theory},
    VOLUME = {1},
      YEAR = {2015},
     PAGES = {Art. 12, 34},
      ISSN = {2363-9555},
   MRCLASS = {11G05 (11F80)},
MRREVIEWER = {\'{A}lvaro Lozano-Robledo},
       DOI = {10.1007/s40993-015-0013-7},
       URL = {https://doi.org/10.1007/s40993-015-0013-7},
}

@article {Greenberg,
    AUTHOR = {Greenberg, Ralph},
     TITLE = {The image of {G}alois representations attached to elliptic
              curves with an isogeny},
   JOURNAL = {Amer. J. Math.},
  FJOURNAL = {American Journal of Mathematics},
    VOLUME = {134},
      YEAR = {2012},
    NUMBER = {5},
     PAGES = {1167--1196},
      ISSN = {0002-9327},
   MRCLASS = {11F80 (11G05)},
  MRNUMBER = {2975233},
MRREVIEWER = {Ravi K. Ramakrishna},
       DOI = {10.1353/ajm.2012.0040},
       URL = {https://doi.org/10.1353/ajm.2012.0040},
}

@article {Balakrishnan,
    AUTHOR = {Balakrishnan, Jennifer and Dogra, Netan and M\"{u}ller, J. Steffen
              and Tuitman, Jan and Vonk, Jan},
     TITLE = {Explicit {C}habauty-{K}im for the split {C}artan modular curve
              of level 13},
   JOURNAL = {Ann. of Math. (2)},
  FJOURNAL = {Annals of Mathematics. Second Series},
    VOLUME = {189},
      YEAR = {2019},
    NUMBER = {3},
     PAGES = {885--944},
      ISSN = {0003-486X},
   MRCLASS = {14G05 (11G18 11G50 11Y50)},
  MRNUMBER = {3961086},
       DOI = {10.4007/annals.2019.189.3.6},
       URL = {https://doi.org/10.4007/annals.2019.189.3.6},
}

@article {Serre81,
    AUTHOR = {Serre, Jean-Pierre},
     TITLE = {Quelques applications du th\'{e}or\`eme de densit\'{e} de {C}hebotarev},
   JOURNAL = {Inst. Hautes \'{E}tudes Sci. Publ. Math.},
  FJOURNAL = {Institut des Hautes \'{E}tudes Scientifiques. Publications
              Math\'{e}matiques},
    NUMBER = {54},
      YEAR = {1981},
     PAGES = {323--401},
      ISSN = {0073-8301},
   MRCLASS = {12A75 (10D99 10H25 14G25)},
  MRNUMBER = {644559},
MRREVIEWER = {J. Tunnell},
       URL = {http://archive.numdam.org/article/PMIHES_1981__54__123_0.pdf},
}

@article {OddDeg,
    AUTHOR = {Bourdon, Abbey and Gill, David R. and Rouse, Jeremy and
              Watson, Lori D.},
     TITLE = {Odd degree isolated points on {$X_1(N)$} with rational
              {$j$}-invariant},
   JOURNAL = {Res. Number Theory},
  FJOURNAL = {Research in Number Theory},
    VOLUME = {10},
      YEAR = {2024},
    NUMBER = {1},
     PAGES = {Paper No. 5, 32},
      ISSN = {2522-0160,2363-9555},
   MRCLASS = {14G35 (11G05)},
  MRNUMBER = {4678892},
       DOI = {10.1007/s40993-023-00488-0},
       URL = {https://doi.org/10.1007/s40993-023-00488-0},
}

@unpublished {OddDegQCurve,
    AUTHOR = {Abbey Bourdon and Filip Najman},
     TITLE = {Sporadic points of odd degree on {$X_1(N)$} coming from $\mathbb{Q}$-curves},
   NOTE = {preprint, available at arxiv.org:2107.10909}
}

@article {Deg3Class,
    AUTHOR = {Derickx, Maarten and Etropolski, Anastassia and van Hoeij,
              Mark and Morrow, Jackson S. and Zureick-Brown, David},
     TITLE = {Sporadic cubic torsion},
   JOURNAL = {Algebra Number Theory},
  FJOURNAL = {Algebra \& Number Theory},
    VOLUME = {15},
      YEAR = {2021},
    NUMBER = {7},
     PAGES = {1837--1864},
      ISSN = {1937-0652},
   MRCLASS = {11G18 (11G05 11Y50 14G35 14H52)},
       DOI = {10.2140/ant.2021.15.1837},
       URL = {https://doi.org/10.2140/ant.2021.15.1837},
}

@article {BPR13,
    AUTHOR = {Bilu, Yuri and Parent, Pierre and Rebolledo, Marusia},
     TITLE = {Rational points on {$X^+_0(p^r)$}},
   JOURNAL = {Ann. Inst. Fourier (Grenoble)},
  FJOURNAL = {Universit\'{e} de Grenoble. Annales de l'Institut Fourier},
    VOLUME = {63},
      YEAR = {2013},
    NUMBER = {3},
     PAGES = {957--984},
      ISSN = {0373-0956},
   MRCLASS = {11G18 (11G05 11G16)},
       DOI = {10.5802/aif.2781},
       URL = {https://doi.org/10.5802/aif.2781},
}

@unpublished {ZywinaImages,
    AUTHOR = {Zywina, David},
     TITLE = {On the possible image of the mod $\ell$ representations associated to elliptic curves over $\mathbb{Q}$},
   JOURNAL = {},
   NOTE = {available at arxiv.org:1508.07660}}

@article {CremonaNajmanQCurve,
    AUTHOR = {Cremona, J. E. and Najman, Filip},
     TITLE = {{$\Bbb Q$}-curves over odd degree number fields},
   JOURNAL = {Res. Number Theory},
  FJOURNAL = {Research in Number Theory},
    VOLUME = {7},
      YEAR = {2021},
    NUMBER = {4},
     PAGES = {Paper No. 62, 30},
      ISSN = {2522-0160},
   MRCLASS = {11G05 (14H52)},
  MRNUMBER = {4314224},
MRREVIEWER = {John T. Cullinan},
       DOI = {10.1007/s40993-021-00270-0},
       URL = {https://doi.org/10.1007/s40993-021-00270-0},
}

@article {LeastCMdegree,
    AUTHOR = {Clark, Pete L. and Genao, Tyler and Pollack, Paul and Saia,
              Frederick},
     TITLE = {The least degree of a {CM} point on a modular curve},
   JOURNAL = {J. Lond. Math. Soc. (2)},
  FJOURNAL = {Journal of the London Mathematical Society. Second Series},
    VOLUME = {105},
      YEAR = {2022},
    NUMBER = {2},
     PAGES = {825--883},
      ISSN = {0024-6107,1469-7750},
   MRCLASS = {11G15 (11G30)},
  MRNUMBER = {4400938},
MRREVIEWER = {Guy\ Fowler},
       DOI = {10.1112/jlms.12518},
       URL = {https://doi.org/10.1112/jlms.12518},
}

@article {LeFournLemos,
    AUTHOR = {Le Fourn, Samuel and Lemos, Pedro},
     TITLE = {Residual {G}alois representations of elliptic curves with
              image contained in the normaliser of a nonsplit {C}artan},
   JOURNAL = {Algebra Number Theory},
  FJOURNAL = {Algebra \& Number Theory},
    VOLUME = {15},
      YEAR = {2021},
    NUMBER = {3},
     PAGES = {747--771},
      ISSN = {1937-0652},
   MRCLASS = {11G05 (11G18)},
  MRNUMBER = {4261100},
       DOI = {10.2140/ant.2021.15.747},
       URL = {https://doi.org/10.2140/ant.2021.15.747},
}

@unpublished{ClarkVolcanoes,
 Author = {Pete L. Clark},
 Title = {{CM} elliptic curves: volcanoes, reality, and applications},
 NOTE = {preprint, available at \url{https://arxiv.org/pdf/2212.13316}}
}

@article {RSZ21,
    AUTHOR = {Rouse, Jeremy and Sutherland, Andrew V. and Zureick-Brown,
              David},
     TITLE = {{$\ell$}-adic images of {G}alois for elliptic curves over
              {$\Bbb{Q}$} (and an appendix with {J}ohn {V}oight)},
      NOTE = {With an appendix with John Voight},
   JOURNAL = {Forum Math. Sigma},
  FJOURNAL = {Forum of Mathematics. Sigma},
    VOLUME = {10},
      YEAR = {2022},
     PAGES = {Paper No. e62, 63},
      ISSN = {2050-5094},
   MRCLASS = {11G05 (11F80 11G18 14G35 14H52)},
  MRNUMBER = {4468989},
MRREVIEWER = {Tristan\ Phillips},
       DOI = {10.1017/fms.2022.38},
       URL = {https://doi.org/10.1017/fms.2022.38},
}

@article {LRcm,
    AUTHOR = {Lozano-Robledo, \'Alvaro},
     TITLE = {Galois representations attached to elliptic curves with
              complex multiplication},
   JOURNAL = {Algebra Number Theory},
  FJOURNAL = {Algebra \& Number Theory},
    VOLUME = {16},
      YEAR = {2022},
    NUMBER = {4},
     PAGES = {777--837},
      ISSN = {1937-0652,1944-7833},
   MRCLASS = {11F80 (11G05 11G15 14H52)},
  MRNUMBER = {4467123},
MRREVIEWER = {Patrick\ Morton},
       DOI = {10.2140/ant.2022.16.777},
       URL = {https://doi.org/10.2140/ant.2022.16.777},
}

@article {Genao21,
    AUTHOR = {Genao, Tyler},
     TITLE = {Typically bounding torsion on elliptic curves isogenous to
              rational {$j$}-invariant},
   JOURNAL = {Proc. Amer. Math. Soc.},
  FJOURNAL = {Proceedings of the American Mathematical Society},
    VOLUME = {151},
      YEAR = {2023},
    NUMBER = {5},
     PAGES = {1907--1914},
      ISSN = {0002-9939,1088-6826},
   MRCLASS = {11G05},
  MRNUMBER = {4556187},
MRREVIEWER = {\"Ozlem\ Ejder},
       DOI = {10.1090/proc/16298},
       URL = {https://doi.org/10.1090/proc/16298},
}

@article {Genao22,
    AUTHOR = {Genao, Tyler},
     TITLE = {Polynomial bounds on torsion from a fixed geometric isogeny
              class of elliptic curves},
   JOURNAL = {J. Th\'eor. Nombres Bordeaux},
  FJOURNAL = {Journal de Th\'eorie des Nombres de Bordeaux},
    VOLUME = {36},
      YEAR = {2024},
    NUMBER = {2},
     PAGES = {661--670},
      ISSN = {1246-7405,2118-8572},
   MRCLASS = {11G05},
  MRNUMBER = {4830946},
MRREVIEWER = {John\ T.\ Cullinan},
       DOI = {10.5802/jtnb.1292},
       URL = {https://doi.org/10.5802/jtnb.1292},
}

@inproceedings {DR73,
    AUTHOR = {Deligne, P. and Rapoport, M.},
     TITLE = {Les sch\'{e}mas de modules de courbes elliptiques},
 BOOKTITLE = {Modular functions of one variable, {II} ({P}roc. {I}nternat.
              {S}ummer {S}chool, {U}niv. {A}ntwerp, {A}ntwerp, 1972)},
    SERIES = {Lecture Notes in Math., Vol. 349},
     PAGES = {143--316},
 PUBLISHER = {Springer, Berlin},
      YEAR = {1973},
   MRCLASS = {14K10 (10D05)},
  MRNUMBER = {0337993},
MRREVIEWER = {T. Oda},
}

@incollection {DiamondIm,
    AUTHOR = {Diamond, Fred and Im, John},
     TITLE = {Modular forms and modular curves},
 BOOKTITLE = {Seminar on {F}ermat's {L}ast {T}heorem ({T}oronto, {ON},
              1993--1994)},
    SERIES = {CMS Conf. Proc.},
    VOLUME = {17},
     PAGES = {39--133},
 PUBLISHER = {Amer. Math. Soc., Providence, RI},
      YEAR = {1995},
   MRCLASS = {11F11 (11F25 11G05 11G18)},
  MRNUMBER = {1357209},
}

@article {Sutherland12,
    AUTHOR = {Sutherland, Andrew V.},
     TITLE = {A local-global principle for rational isogenies of prime
              degree},
   JOURNAL = {J. Th\'{e}or. Nombres Bordeaux},
  FJOURNAL = {Journal de Th\'{e}orie des Nombres de Bordeaux},
    VOLUME = {24},
      YEAR = {2012},
    NUMBER = {2},
     PAGES = {475--485},
      ISSN = {1246-7405,2118-8572},
   MRCLASS = {11G05 (14K15)},
  MRNUMBER = {2950703},
MRREVIEWER = {Joseph\ H.\ Silverman},
       URL = {http://jtnb.cedram.org/item?id=JTNB_2012__24_2_475_0},
}

@unpublished {LMFDB,
    AUTHOR = {The LMFDB Collaboration},
     TITLE = {The {L}-functions and Modular Forms Database, (2023)},
   JOURNAL = {},
   NOTE = {Available at lmfdb.org},
}

@unpublished{FurioLombardo23,
    AUTHOR = {Furio, Lorenzo and Lombardo, Davide},
     TITLE = {Serre's uniformity question and proper subgroups of {$C_{ns}^+(p)$}},
  NOTE = {preprint, available at arXiv:2305.17780},
}

@unpublished{FurioLombardo25,
    AUTHOR = {Furio, Lorenzo and Lombardo, Davide},
     TITLE = {On 7-adic {G}alois representations for elliptic curves over {$\mathbb{Q}$}},
  NOTE = {preprint, available at arXiv:2305.17780},
}

@article {QuadPointsConjecture2024,
    AUTHOR = {Ad\v zaga, Nikola and Keller, Timo and Michaud-Jacobs,
              Philippe and Najman, Filip and Ozman, Ekin and Vukorepa,
              Borna},
     TITLE = {Computing quadratic points on modular curves {$X_0(N)$}},
   JOURNAL = {Math. Comp.},
  FJOURNAL = {Mathematics of Computation},
    VOLUME = {93},
      YEAR = {2024},
    NUMBER = {347},
     PAGES = {1371--1397},
      ISSN = {0025-5718,1088-6842},
   MRCLASS = {11G05 (11G18 14G05)},
  MRNUMBER = {4708039},
MRREVIEWER = {Tyler\ Genao},
       DOI = {10.1090/mcom/3902},
       URL = {https://doi.org/10.1090/mcom/3902},
}

@incollection {Magma,
    AUTHOR = {Bosma, Wieb and Cannon, John and Playoust, Catherine},
     TITLE = {The {M}agma algebra system. {I}. {T}he user language},
      NOTE = {Computational algebra and number theory (London, 1993)},
   JOURNAL = {J. Symbolic Comput.},
  FJOURNAL = {Journal of Symbolic Computation},
    VOLUME = {24},
      YEAR = {1997},
    NUMBER = {3-4},
     PAGES = {235--265},
      ISSN = {0747-7171},
   MRCLASS = {68Q40},
  MRNUMBER = {1484478},
       DOI = {10.1006/jsco.1996.0125},
       URL = {https://doi.org/10.1006/jsco.1996.0125},
}

@unpublished {classificationdeg4points,
    AUTHOR = {Maarten Derickx and Filip Najman},
     TITLE = {Classification of torsion of elliptic curves over quartic fields},
   NOTE = {preprint, available at arxiv.org:2412.16016}
}

@misc{ogg,
      title={Ogg's Torsion conjecture: Fifty years later}, 
      author={Jennifer S. Balakrishnan and Barry Mazur},
      year={2024},
      eprint={2307.04752},
      archivePrefix={arXiv},
      primaryClass={math.NT},
      url={https://arxiv.org/abs/2307.04752}, 
}

@unpublished {Furio2024,
    AUTHOR = {Lorenzo Furio},
     TITLE = {Effective bounds for adelic Galois representations attached to elliptic curves over the rationals},
   NOTE = {preprint, available at arxiv.org:2412.10340},
}

@unpublished {BourdonEjder,
    AUTHOR = {Bourdon, Abbey and \"{O}zlem Ejder},
     TITLE = {Rational isolated $j$-invariants from {$X_1(\ell^n)$} and {$X_0(\ell^n)$}},
   NOTE = {preprint, available at arxiv.org:2506.19560},
}
\end{document}